\documentclass[11pt]{article}

\usepackage{amsmath}
\usepackage{amsthm}
\usepackage{amsfonts}
\usepackage{bbm}
\usepackage{amssymb}
\usepackage{graphicx}
\usepackage{color}

\usepackage{blindtext}
\usepackage{enumitem}

\usepackage{fancybox}
\newlength{\querylen}
\newcommand{\mmp}{\mathbb{P}}

\newcommand{\od}{\overset{{\rm d}}{=}}

\newcommand{\me}{\mathbb{E}}
\newcommand{\E}{\mathbb{E}}

\newcommand{\mr}{\mathbb{R}}
\newcommand{\mn}{\mathbb{N}}

\DeclareMathOperator{\1}{\mathbbm{1}}

\newtheorem{thm}{Theorem}[section]
\newtheorem{lemma}[thm]{Lemma}

\newtheorem{cor}[thm]{Corollary}

\newtheorem{assertion}[thm]{Proposition}
\theoremstyle{definition}

\theoremstyle{remark}
\newtheorem{rem}[thm]{Remark}

\begin{document}
\title{On intermediate levels of nested occupancy scheme in random environment generated by stick-breaking I}

\author{Dariusz Buraczewski\footnote{Mathematical Institute, University of
Wroc{\l}aw, Wroc{\l}aw, Poland; e-mail: dbura@math.uni.wroc.pl}, \ Bohdan Dovgay\footnote{Faculty of Computer Science and
Cybernetics, Taras Shevchenko National University of Kyiv, Kyiv, Ukraine; \ e-mail: bogdov@gmail.com} \
and \ Alexander Iksanov\footnote{Faculty of Computer Science and
Cybernetics, Taras Shevchenko National University of Kyiv, Kyiv, Ukraine; \ e-mail: iksan@univ.kiev.ua}}

\maketitle
\begin{abstract}
\noindent Consider a weighted branching process generated by the lengths of intervals obtained by stick-breaking of unit length (a.k.a.\ the residual allocation model) and associate with each weight a `box'. Given the weights `balls' are thrown independently into the boxes of the first generation with probability of hitting a box being equal to its weight. Each ball located in a box of the $\ell$th generation, independently of the others, hits a daughter box in the $(\ell+1)$th generation with probability being equal the ratio of the daughter weight and the mother weight. This is what we call nested occupancy scheme in random environment. Restricting attention to a particular generation one obtains the classical Karlin occupancy scheme in random environment.

Assuming that the stick-breaking factor has a uniform distribution on $[0,1]$ and that the number of balls is $n$ we investigate occupancy of intermediate generations, that is, those with indices $\lfloor j_n u\rfloor$ for $u>0$, where $j_n$ diverges to infinity at a sublogarithmic rate as $n$ becomes large. Denote by $K_n(j)$ the number of occupied (ever hit) boxes in the $j$th generation. It is shown that the finite-dimensional distributions of the process $(K_n(\lfloor j_n u\rfloor))_{u>0}$, properly normalized and centered, converge weakly to those of an integral functional of a Brownian motion. The case of a more general stick-breaking is also analyzed.
\end{abstract}


\hphantom{2000 Mathematics Subject Classification: }Secondary:
60C05

\section{Introduction}\label{Sect1}

The infinite occupancy scheme is obtained by allocating `balls'
independently over an infinite collection of `boxes' $1$,
$2,\ldots$ with probability $p_r$ of hitting box $r$, $r\in\mn$,
where $\sum_{r\geq 1}p_r=1$. The first articles in which such a
scheme was investigated are \cite{Bahadur:1960, Darling:1967,
Karlin:1967}. Sometimes the infinite occupancy scheme is called
Karlin's occupancy scheme because of Karlin's seminal contribution
\cite{Karlin:1967}. Surveys on the infinite occupancy can be found
in \cite{Barbour+Gnedin:2009, Gnedin+Hansen+Pitman:2007}, see also
p.~207 in \cite{Iksanov:2016}.

The infinite occupancy scheme in a random environment has received
much less attention. The definition of the latter schemes assumes
that probabilities $(P_r)_{r\in\mn}$ are random with an arbitrary joint distribution satisfying $\sum_{r\geq
1}P_r=1$ a.s., and that conditionally given $(P_r)_{r\in\mn}$
`balls' are allocated independently with probability $P_r$ of
hitting box $r$, $r\in\mn$. Known examples of such randomized
infinite occupancy schemes correspond to the environments
(probabilities) given by a residual allocation model
\begin{equation}\label{BS}
P_r:=W_1W_2\cdot\ldots\cdot W_{r-1}(1-W_r),\quad r\in\mn,
\end{equation}
where
\begin{description}[align=left]
\item [(A)] $W_1$, $W_2,\ldots$ are independent copies of a random variable $W$ taking values in $(0,1)$;
\item [(B)] $W_1$, $W_2,\ldots$ are independent random variables such that $W_r$ has a beta distribution with parameters $\alpha r$ and $1$ for
$r\in\mn$ and $\alpha>0$, that is, $\mmp\{W_r\in {\rm d}x\}=\alpha
r x^{\alpha r-1}\1_{(0,1)}(x){\rm d}x$.
\item [(C)] $W_1$, $W_2,\ldots$ are independent random variables such that $W_r$ has a beta distribution with parameters
$\theta+\alpha r$ and $1-\alpha$ for $r\in\mn$, $\alpha\in (0,1)$
and $\theta>-\alpha$, that is, $\mmp\{W_r\in {\rm d}x\}= (1/{\rm
B}(\theta+\alpha r, 1-\alpha)) x^{\theta+\alpha
r-1}(1-x)^{-\alpha}\1_{(0,1)}(x){\rm d}x$, where ${\rm B}(\cdot,
\cdot)$ is the beta function.
\end{description}
In the case C the distribution of $(P_1, P_2,\ldots)$ is called
{\rm GEM} (Griffiths-Engen-McCloskey) distribution with parameters
$\alpha$ and $\theta$ (${\rm GEM}(\alpha,\theta)$). If in the case
A the variable $W$ is beta distributed with parameters $\theta$
and $1$, then the distribution of $(P_1, P_2,\ldots)$ is ${\rm
GEM}(0,\theta)$.

The random occupancy scheme arising in the case A is called
Bernoulli sieve. This scheme was introduced in \cite{Gnedin:2004}
and investigated in many articles. The picture up to 2010 is
surveyed in \cite{Gnedin+Iksanov+Marynych:2010b}. The state of the
art concerning several aspects of the Bernoulli sieve is
summarized in Chapter 5 of \cite{Iksanov:2016}. More recent
contributions include \cite{Alsmeyer+Iksanov+Marynych:2017,
Duchamps+Pitman+Tang:2017+, Iksanov+Jedidi+Bouzeffour:2017,
Pitman+Tang:2017+}. The random occupancy schemes arising in the
cases B and C were partially investigated in
\cite{Robert+Simatos:2009} and \cite{Favaro etal:2018, Pitman+Yakubovich:2017,
Pitman+Yakubovich:2017b} and many other papers, respectively.

Another popular model
(\cite{Barbour+Gnedin:2006,Gnedin+Iksanov:2012,Gnedin+Pitman+Yor:2006,Gnedin+Pitman+Yor:2006a})
of the infinite occupancy scheme in random environment assumes
that the probabilities $(P_r)_{r\in\mn}$ are formed by an
enumeration of the a.s.\ positive points of
\begin{equation}\label{levy}
\{e^{-X(t-)}(1-e^{-\Delta X(t)}): t\geq 0\},
\end{equation}
where $X:=(X(t))_{t\geq 0}$ is a subordinator (that is, a nondecreasing L\'{e}vy process) with $X(0)=0$, zero drift, no
killing and a nonzero L\'{e}vy measure, and $\Delta X(t)$ is a
jump of $X$ at time $t$. Since the closed range of the process $X$
is a regenerative subset of $[0,\infty)$ of zero Lebesgue measure,
we have $\sum_{r\geq 1}P_r=1$ a.s. Note that collection
\eqref{levy} transforms into \eqref{BS} when $X$ is a compound
Poisson process. However, when the L\'{e}vy measure of $X$ is
infinite, collections \eqref{levy} and \eqref{BS} are principally
different.

Following \cite{Bertoin:2008} and \cite{Gnedin+Iksanov:2020} (see also \cite{Joseph:2011,
Businger:2017}) we now define a nested infinite occupancy scheme in random environment. This means that
we construct a nested sequence of the environments $(P_k)$ and the
corresponding `boxes' so that the same collection of `balls' is
thrown into all `boxes'. To this end, we work with a weighted
branching process with positive weights which is nothing else but
a multiplicative counterpart of a branching random walk.

Let $\mathbb{V}=\cup_{n\in\mn_0}\mn^n$ be the set of all possible
individuals of some population, where $\mn_0:=\mn\cup\{0\}$. The ancestor is identified with
the empty word $\varnothing$ and its weight is $P(\varnothing)=1$.
On some probability space $(\Omega, \mathcal{F}, \mmp)$ let
$((P_r(v))_{r\in\mn})_{v \in \mathbb{V}}$ be a family of
independent copies of $(P_r)_{r\in\mn}$. An individual $v
=v_1\ldots v_n$ of the $n$th generation whose weight is denoted by
$P(v)$ produces an infinite number of offspring residing in the $(n+1)$th generation. The
offspring of the individual $v$ are enumerated by $vr = v_1 \ldots
v_n r$, where $r\in \mn$, and the weights of the offspring are
denoted by $P(vr)$. It is postulated that $P(vr)=P(v)P_r(v)$.
Observe that, for each $j\in\mn$, $\sum_{|v|=j} P(v)=1$ a.s.,
where, by convention, $|v|=j$ means that the sum is taken over all
individuals of the $\ell$th generation. For $j\in\mn$, denote by
$\mathcal{F}_j$ the $\sigma$-algebra generated by
$(P(v))_{|v|=1},\ldots, (P(v))_{|v|=j}$. The nested sequence of
environments is formed by the weights of the subsequent
generations individuals, that is, $(P(v))_{|u|=1}$,
$(P(v))_{|v|=2},\ldots$. Further, we identify individuals with
`boxes'. At time $j=0$, all `balls' are collected in the box
$\varnothing$. At time $j=1$, given $\mathcal{F}_1$, `balls' are
allocated independently with probability $P(v)$ of hitting box
$v$, $|v|=1$. At time $j=k$, given $\mathcal{F}_k$, a ball located
in the box $v$ with $|v|=k$ is placed independently of the others
into the box $vr$ with probability $P_r(v)=P(vr)/P(v)$.

From purely mathematical viewpoint, the nested occupancy schemes in random environment are interesting because these include features of the infinite occupancy schemes
and the weighted branching processes. Even though the latter two objects are rather popular, they belong to distinct areas of probability theory. Furthermore, the combined
model should exhibit new effects which are not seen in the individual components. Turning to possible applications recall that the number of occupied boxes in the given
generation of the weighted branching process can be interpreted as the number of types present in a sample of some ‘individuals’. Assuming that the type of the individual is
inherited by her progeny the numbers of occupied boxes in subsequent generations of the nested occupancy scheme model the evolution of the number of various types of
some population as time elapses.

Assume that there are $n$ balls. For $r=1,2,\ldots, n$ and
$j\in\mn$, denote by $K_{n,j,r}$ the number of boxes in the $j$th
generation which contain exactly $r$ balls and set
\begin{equation*}
K_n^{(s)}(j):=\sum_{r=\lceil n^{1-s}\rceil}^n K_{n,j,r},\quad s\in [0,1],
\end{equation*}
where $x\mapsto \lceil x \rceil=\min\{n\in\mathbb{Z}: n\geq x\}$
is the ceiling function. Observe that $K_n(j):=K^{(1)}_n(j)$ is the
number of occupied boxes in the $j$th generation.

For $j\in\mn$ and $t>0$, denote by $\rho_j(t):=\#\{|u|=j: P(u)\geq 1/t\}$ the counting function for the probabilities in the $j$th generation. The earlier investigations of the occupancy schemes in random environment by J. Bertoin \cite{Bertoin:2008}, A. Joseph  \cite{Joseph:2011} and S. Businger \cite{Businger:2017}) focused on the behavior near the boundary of a weighted branching process tree. This setting enabled the authors to make use of various asymptotic properties of the weighted branching process such as the convergence of the Biggins martingale and large deviations. In particular, the information concerning the allocation of balls at early generations was not needed, for it just got lost on the way from the root to the boundary. Businger (in the setting of multitype weighted branching processes) and Joseph were interested in (among other problems) the a.s.\ convergence as $n\to\infty$ of $\inf\{j\in\mn: K_n(j)=n\}$, that is, the index of the earliest generation in which all boxes contain no more than one ball. Alternatively, this can be thought of as the height of the weighted branching process subtree within which the nondecreasing sequence $(K_n(j))_{j\in\mn}$ stays away from $n$. These authors did not investigate weak convergence of the number of occupied boxes at all. Bertoin in Theorem 1 of \cite{Bertoin:2008} proved that $K_n(j_n)/f(n)$ converges almost surely to a limit random variable which includes as a factor the terminal value of the Biggins martingale (it is the presence of the Biggins martingale limit which manifests the influence of the boundary). Here, $f(x)$ is a regularly varying function which is given explicitly, and $(j_n)_{n\in\mn}$ is the sequence of integers satisfying $\lim_{n\to\infty}(j_n-a\log n)=b$ for some $a>0$ and $b\in\mr$. In \cite{Gnedin+Iksanov:2020} the behavior near the root of a weighted branching process tree was investigated. More precisely, the occupancy of any fixed number of fixed (not depending on $n$) generations was analyzed. In this setting the well-developed asymptotic machinery available for the weighted branching processes is useless, and one should rather exploit the recursive structure of the weighted branching process in combination with a precise information about the random environment in the first generation. It turns out the amount of randomness which is present in the first generation remains unchanged in the subsequent fixed generations. Indeed, if the counting function of probabilities in the first generation $\rho_1(t)$, properly normalized, centered and rescaled, converges weakly as $t\to\infty$ to a Gaussian process $G$, say, then the counting function of probabilities in the $j$th generation $\rho_j(t)$, again normalized, centered and rescaled, weakly converges to an integral functional $f_j$ of $G$. The form of the functional is determined by the weighted branching process tree.

In view of the preceding discussion a natural question is: at which generations of the tree is there a phase transition between the root dominance and the boundary dominance? This motivates us to investigate, for the first time, the occupancy in the $j$th generations for $j=j_n \to\infty$ and $j_n=o(\log n)$ as $n\to\infty$. Our results will show (at least for a particular environment) that these {\it intermediate} generations belong to the range of the root dominance, thereby answering the aforementioned question: the phase transition occur at generations with indices $j=j_n\sim c\log n$ for some $c>0$. Throughout the rest of the paper we assume that the distribution of the environment in the first generation is ${\rm GEM}(0,1)$ or follows a more general residual allocation model. Such a restriction is caused by two reasons. First, calculations required become messy even in these particular cases. Second, our purpose is to give a feeling of what is happening in the intermediate generations rather than investigate the occupancy of these generations for the most general environment

Finally, we provide an informal justification of the phase transition under the assumption that the random probabilities in the first generation exhibit an exponential-like decay $P_k\asymp \exp(-k^\rho)$ as $k\to\infty$ for some $\rho>0$. Roughly speaking, the latter is more or less equivalent to $\me \rho_1(t)\asymp (\log t)^{1/\rho}$ as $t\to\infty$. The functions $\me \rho_j(t)$ are still slowly varying for integers $j=o(\log t)$, whereas these become regularly varying for $j\sim c\log t$. One may expect that similarly to the known phenomenon for the infinite occupancy in random environment (that is, when attention is restricted to the first generation) the regular variation of the counting function of probabilities entails the a.s.\ convergence of the number of occupied boxes, properly normalized, whereas the slow variation entails distributional fluctuations of the number of occupied boxes, properly normalized and centered.

\section{Main result}\label{main345}

Throughout the paper we write  $\Rightarrow$, ${\overset{{\rm
d}}\longrightarrow}$ and ${\overset{{\rm f.d.d.}}\longrightarrow}$
to denote weak convergence in a function space, weak convergence
of one-dimensional and finite-dimensional distributions,
respectively.

Assume that the distribution of $(P_1, P_2,\ldots)$ is ${\rm GEM}(0,1)$. This means that \eqref{BS} holds with independent $W_k$'s having a uniform distribution on $[0,1]$, that is, $\mmp\{W_k\in{\rm d}x\}=\1_{(0,1)}(x){\rm d}x$. In \cite{Gnedin+Iksanov:2020}, for a rather wide class of exponentially decaying random probabilities (environments), a functional limit theorem
was proved for the random process $\big(K_n^{(u)}(1),\ldots, K_n^{(u)}(\ell)\big)_{u\geq 0}$ for each $\ell\in\mn$, properly
normalized and centered, as $n\to\infty$. In particular, according to Corollary 4.3 in \cite{Gnedin+Iksanov:2020} the following multivariate central limit theorem holds in the ${\rm GEM}(0,1)$ case: for each $\ell\in\mn$, as $n\to\infty$,
\begin{equation}\label{multivariate}
\Bigg(\frac{(j-1)!\big(K_n(j)-(\log n)^j/j!\big)}{(\log n)^{j-1/2}}\Bigg)_{j=1,\ldots,\ell} ~{\overset{{\rm
d}}\longrightarrow}~ (N_1,\ldots,N_\ell),
\end{equation}
where $(N_1,\ldots,N_\ell)$ is a normal random vector
with zero mean and covariances
\begin{equation}\label{eq:cov_Ni_Nj}
\me N_i N_j = \frac{1}{i+j-1}, \qquad 1\leq i,j \leq \ell.
\end{equation}
Let $(j_n)_{n\in\mn}$ be a sequence of positive numbers satisfying
$j_n\to \infty$ and $j_n=o(\log n)$ as $n\to\infty$. Our purpose
is to investigate weak convergence of the process $(K_n(\lfloor j_n
u\rfloor))_{u>0}$, again properly normalized and centered. Here and hereafter, $x\mapsto \lfloor x \rfloor=\max\{n\in\mathbb{Z}: n\leq x\}$ is the floor function. This provides information about occupancy of {\it intermediate levels} in the infinite occupancy scheme in random environment generated by the uniform stick-breaking, hence the title of the paper. Here is our main result.
\begin{thm}\label{main3}
Assume that $(P_1, P_2,\ldots)$ has the ${\rm GEM}(0,1)$ distribution. Let $(j_n)_{n\in\mn}$ be a sequence of positive numbers satisfying
$j_n\to \infty$ and $j_n=o(\log n)$ as $n\to\infty$. The following
limit theorem holds, as $n\to\infty$,
\begin{equation}\label{clt55}
\Bigg(\frac{\lfloor j_n\rfloor^{1/2}(\lfloor j_n u\rfloor-1)!\big(K_n(\lfloor j_n u\rfloor)-(\log n)^{\lfloor j_n u\rfloor}/\lfloor j_n u\rfloor!\big)}{(\log n)^{\lfloor j_n
u\rfloor-1/2}}\Bigg)_{u>0}~{\overset{{\rm f.d.d.}}\longrightarrow}~\Bigg(\int_{[0,\,\infty)}e^{-uy}{\rm
d}B(y)\Bigg)_{u>0},
\end{equation}
where $(B(v))_{v\geq 0}$ is a standard Brownian motion.
\end{thm}
\begin{rem}
The limit process in Theorem \ref{main3} can be defined via
integration by parts $$R(u):=\int_{[0,\,\infty)}e^{-uy}{\rm
d}B(y)=u\int_0^\infty e^{-uy}B(y){\rm d}y,\quad u>0.$$ The process
$(R(u))_{u>0}$ is a.s.\ continuous on $(0,\infty)$. However, it cannot be
defined by continuity at $u=0$ because of the oscillating behavior
of the Brownian motion at $\infty$. This explains that the limit
theorem holds for $u>0$ rather than $u\geq 0$.
\end{rem}
\begin{rem}
A perusal of the proof of Theorem \ref{main3} reveals that the Brownian motion $(B(u))_{u\geq 0}$ appearing in the limit is the same as the weak limit for $(\rho_1(e^{ut}))_{u\geq 0}$, properly normalized and centered. The function $y\mapsto e^{-uy}$ comes from the renewal structure of the weighted branching process tree. This justifies the claim made in the introduction that the intermediate generations belong to the range of the root dominance.
\end{rem}

It can be checked that $R(u)$ has the same distribution as $(2u)^{-1/2}B(1)$ for each $u>0$. As a consequence, we obtain a one-dimensional central limit theorem when taking in \eqref{clt55} $u=1$.
\begin{cor}
In addition to the assumptions of Theorem \ref{main3} let $(j_n)$ be a sequence of positive integers. Then, as $n\to\infty$,
$$\frac{(2 j_n)^{1/2}(j_n-1)! \big(K_n(j_n)-(\log n)^{j_n}/j_n!\big)}{(\log n)^{j_n-1/2}}~{\overset{{\rm
d}}\longrightarrow}~ B(1).$$
\end{cor}

Now we explain how the result of Theorem \ref{main3} can be guessed from ~\eqref{multivariate}. To this end, we note that
$$\me R(u)R(v)=(u+v)^{-1},\quad u,v>0.$$ Taking $i=\lfloor j_n u\rfloor$ and $j=\lfloor j_n v\rfloor$ in~\eqref{eq:cov_Ni_Nj}
we obtain the covariance $1/(\lfloor j_nu\rfloor+ \lfloor j_nv\rfloor-1)\sim (u+v)^{-1}j_n^{-1}$ as $n\to\infty$, where $x_n\sim y_n$ as $n\to\infty$
means that $\lim_{n\to\infty}(x_n/y_n)=1$.

A slight deviation from the ${\rm GEM}(0,1)$ case complicates significantly both the essence of the problem and technicalities. Allowing for a more general stick-breaking factor $W$ in \eqref{BS} and employing a non-principal modification of the proof of Theorem \ref{main3} we can only cover intermediate generations in the range $j_n=o((\log n)^{1/3})$.
\begin{thm}\label{main100}
Let $P_1$, $P_2,\ldots$ be given by \eqref{BS} with independent and identically distributed $W_1$, $W_2,\ldots$ taking values in $(0,1)$. Assume that the distribution of $|\log W|$ is nonlattice, that $\sigma^2:={\rm Var}(\log W)\in (0,\infty)$ and that $\me |\log(1-W)|<\infty$. Let $(j_n)_{n\in\mn}$ be a sequence of positive numbers satisfying
$j_n\to \infty$ and $j_n=o((\log n)^{1/3})$ as $n\to\infty$. The following
limit theorem holds, as $n\to\infty$,
\begin{equation}\label{clt5555}
\Bigg(\frac{\lfloor j_n\rfloor^{1/2}(\lfloor j_n u\rfloor-1)!\big(K_n(\lfloor j_n u\rfloor)-(\mu^{-1}\log
n)^{\lfloor j_n u\rfloor}/\lfloor j_n u\rfloor!\big)}{(\sigma^2 \mu^{-2\lfloor j_n u\rfloor-1} (\log n)^{2\lfloor j_n
u\rfloor-1})^{1/2}}\Bigg)_{u>0}~{\overset{{\rm f.d.d.}}\longrightarrow}~\Bigg(\int_{[0,\,\infty)}e^{-uy}{\rm
d}B(y)\Bigg)_{u>0},
\end{equation}
where $(B(v))_{v\geq 0}$ is a standard Brownian motion and $\mu:=\me |\log W|<\infty$.
\end{thm}
\begin{rem}
The referee has kindly informed us that replacing in \eqref{clt5555} $\lfloor j_n \rfloor^{1/2}$ by $\lfloor j_n u \rfloor^{1/2}$ would lead to a limit process $(\bar R(u))_{u\geq 0}$ given by   $\bar R(u): = u^{1/2}R(u)$. Such an alternative formulation would offer an additional bonus which consists in stationarity of $(\bar R(e^u))_{u\in\mr}$. Replacing in this alternative formulation
$\lfloor j_n u \rfloor$ by $\lfloor j_n u \rfloor+1$ would make Theorem \ref{main100} hold for $u=0$, as well. The same remarks apply equally well to Theorem \ref{main3}.
\end{rem}

We believe that a central limit theorem with a more complicated centering still holds in the range ${\rm const}\,(\log n)^{1/3}\leq j_n=o((\log n)^{1/2})$. However, we have no reasonable conjecture for the range ${\rm const}\,(\log n)^{1/2}\leq j_n=o(\log n)$.

\section{Limit theorems for a special branching random walk}

It is more convenient to work with a branching random walk which is an additive counterpart of the original weighted branching process obtained by the logarithmic transformation.

The proofs of Theorems \ref{main3} and \ref{main100} are based on several intermediate results, and the three of these will be deduced from limit theorems for a special branching random walk. Let $(\xi_i, \eta_i)_{i\in\mn}$ be independent copies of a random vector $(\xi, \eta)$ with positive arbitrarily dependent components. Denote by $(S_i)_{i\geq 0}$ the zero-delayed
ordinary random walk with increments $\xi_i$ for $i\in\mn$, that is, $S_0:=0$ and $S_i:=\xi_1+\ldots+\xi_i$ for $i\in\mn$. Define
\begin{equation*}
T_i:=S_{i-1}+\eta_i,\quad i\in \mn.
\end{equation*}
The sequence $T:=(T_i)_{i\in\mn}$ is called {\it perturbed random walk}. A survey of various results for the so defined perturbed random walks can be found in the book \cite{Iksanov:2016}. Put $N(t):=\sum_{i\geq 1}\1_{\{T_i\leq t\}}$ and $V(t):=\me N(t)$ for $t\geq 0$. It is clear that
\begin{equation}\label{equ}
V(t)=\me U((t-\eta)^+)=\int_{[0,\,t]}U(t-y){\rm d}G(y),\quad t\geq 0
\end{equation}
where, for $t\geq 0$, $U(t):=\sum_{i\geq 0}\mmp\{S_i\leq t\}$ is
the renewal function and $G(t)=\mmp\{\eta\leq t\}$.

Assuming that the distribution of $\xi$ is nonlattice and ${\tt m}:=\me \xi<\infty$, let $\bar S_0$ be a random variable with distribution
$$\mmp\{\bar S_0\in {\rm d}x\}={\tt m}^{-1}\mmp\{\xi>x\}\1_{(0,\infty)}(x){\rm d}x$$ which is independent of $(\xi_i, \eta_i)_{i\in\mn}$. Define a stationary version of $T$ by
\begin{equation*}
\bar T_i:=\bar S_0+ T_i=\bar S_0+ S_{i-1}+\eta_i=:\bar S_{i-1}+\eta_i,\quad i\in \mn.
\end{equation*}
Put $\bar N(t):=\sum_{i\geq 1}\1_{\{\bar T_i\leq t\}}$ and $\bar V(t):=\me \bar N(t)$ for $t\geq 0$. Since $\sum_{i\geq 0}\mmp\{\bar S_i\leq t\}={\tt m}^{-1}t$ for $t\geq 0$ (see, for instance, formula (3.7) on p.~91 in \cite{Iksanov:2016}), we infer
\begin{equation}\label{equsta}
\bar V(t)={\tt m}^{-1}\int_{[0,\,t]}(t-y){\rm d}G(y)={\tt m}^{-1}\int_0^t G(y){\rm d}y,\quad t\geq 0.
\end{equation}

Let $T^\ast:=(T_i^\ast)_{i\in\mn}$ denote both $T$ and $\bar T:=(\bar T_i)_{i\in\mn}$. We are now ready to recall the construction of a branching random walk in the special case it is generated by $T^\ast$. At time $0$ there is one individual, the ancestor. The ancestor produces offspring (the first generation) with positions on $[0,\infty)$ given by the points of $T^\ast$. The
first generation produces the second generation. The displacements of positions of the second generation individuals with respect to their mothers' positions are distributed according to
copies of $T^\ast$, and for different mothers these copies are independent. The second generation produces the third one, and so on. All individuals act independently of each other.

In what follows we use $\ast$ to denote the quantities pertaining to both $T$ and $\bar T$. For $t\geq 0$ and $j\in\mn$, denote by $N^\ast_j(t)$ the number of the $j$th generation individuals with positions $\leq t$ and put $V^\ast_j(t):=\me N^\ast_j(t)$. Then $N^\ast_1(t)=N^\ast(t)$, $V^\ast_1(t)=V^\ast(t)$ and $$V^\ast_j(t)=\int_{[0,\,t]} V^\ast_{j-1}(t-y){\rm d}V^\ast(y),\quad j\geq 2,~ t\geq 0.$$
The basic decomposition that we need reads
\begin{equation}\label{basic1232}
N^\ast_j(t)=\sum_{r\geq 1}N^{(\ast,\, r)}_{j-1}(t-T^\ast_r)\1_{\{T^\ast_r\leq t\}},\quad j\geq 2,\, t\geq 0,
\end{equation}
where $N_{j-1}^{(\ast, r)}(t)$ is the number of successors in the $j$th generation which reside in the interval $[T^\ast_r,t+T^\ast_r]$ of the first generation individual with position $T^\ast_r$. By the branching property, $(N_{j-1}^{(\ast, 1)}(t))_{t\geq 0}$,
$(N_{j-1}^{(\ast, 2)}(t))_{t\geq 0},\ldots$ are independent copies of
$(N^\ast_{j-1}(t))_{t\geq 0}$ which are also independent of $T^\ast$.  Note that, for
$j\geq 2$, $(N^\ast_j(t))_{t\geq 0}$ is a particular instance of a
random process with immigration at random times (the term was introduced in
\cite{Dong+Iksanov:2020}, see also \cite{Iksanov+Rashytov:2020}).

The following observations are needed for the proof of Theorem \ref{main100}.
\begin{assertion}\label{1formain100}
Assume that the distribution of $\xi$ is nonlattice, that ${\tt s}^2:={\rm Var}\,\xi \in (0,\infty)$ and that $\me \eta<\infty$. Let $t\mapsto j(t)$ be any positive function satisfying $j(t)\to \infty$ and $j(t)=o(t^{1/3})$ as $t\to\infty$. Then
\begin{equation}\label{inter1}
\bigg(\frac{\lfloor j(t)\rfloor^{1/2}(\lfloor j(t)u\rfloor-1)!}{{\tt m}^{-\lfloor j(t)u\rfloor}t^{\lfloor j(t)u\rfloor-1/2}}\big|N_{\lfloor j(t)u\rfloor}(t)-\bar N_{\lfloor j(t)u\rfloor}(t)\big|\bigg)_{u>0}~{\overset{{\rm f.d.d.}}\longrightarrow}~(\Theta(u))_{u>0},
\end{equation}
where $\Theta(u):=0$ for $u>0$ and ${\tt m}=\me\xi<\infty$, and
\begin{eqnarray}
&&\left(\frac{\lfloor j(t)\rfloor^{1/2}(\lfloor j(t)u\rfloor-1)!}{{\tt m}^{-\lfloor j(t)u\rfloor}t^{\lfloor j(t)u\rfloor-1/2}}\bigg|\sum_{r\geq 1}\Big(\frac{(t-\bar T_r)^{\lfloor j(t)u\rfloor-1}}{(\lfloor j(t)u\rfloor-1)!{\tt m}^{\lfloor j(t)u\rfloor-1}}-\bar V_{\lfloor j(t)u\rfloor-1}(t-\bar T_r)\Big)\1_{\{\bar T_r\leq t\}}\bigg|\right)_{u>0}\notag\\
&~{\overset{{\rm f.d.d.}}\longrightarrow}~& (\Theta(u))_{u>0}.\label{inter2}
\end{eqnarray}
\end{assertion}

Theorems \ref{main4} and \ref{main5} are important ingredients for the proofs of Theorems \ref{main3} and \ref{main100}. As usual, $\od$ denotes equality of distributions.
\begin{thm}\label{main4}
Let $t\mapsto j(t)$ be any positive function satisfying $\lim_{t\to\infty} j(t)=\infty$. Assume that either
\begin{equation}\label{equa}
(\xi,\eta)\od (|\log W|, |\log(1-W)|),
\end{equation}
where $W$ has a uniform distribution on $[0,1]$, and $j(t)=o(t)$ as $t\to\infty$, or the distribution of $\xi$ is nonlattice, ${\tt s}^2={\rm Var}\,\xi \in (0,\infty)$, $\me \eta<\infty$, and $j(t)=o(t^{1/2})$ as $t\to\infty$. Then, as $t\to\infty$, $$\bigg(\frac{\lfloor j(t)\rfloor^{1/2}(\lfloor j(t)u\rfloor-1)!}{{\tt m}^{-\lfloor j(t)u\rfloor} t^{\lfloor j(t)u\rfloor-1/2}}\Big(N^\ast_{\lfloor j(t)u\rfloor}(t)-\sum_{r\geq 1}V^\ast_{\lfloor j(t)u\rfloor-1}(t-T^\ast_r)\big)\1_{\{T^\ast_r\leq t\}}\bigg)_{u>0}~{\overset{{\rm f.d.d.}}\longrightarrow}~(\Theta(u))_{u>0},$$ where $\Theta(u)=0$ for $u>0$ and ${\tt m}=\me\xi<\infty$ (${\tt m}=1$ when \eqref{equa} prevails).
\end{thm}

In what follows we denote by $D(0,\infty)$ ($D[0,\infty)$) the
Skorokhod space of right-continuous functions defined on
$(0,\infty)$ (on $[0,\infty)$) with finite limits from the left at positive points.
\begin{thm}\label{main5}
Assume that ${\tt s}^2={\rm Var}\,\xi\in (0,\infty)$ and that $\me \eta<\infty$. Let $t\mapsto j(t)$ be any positive function satisfying $j(t)\to \infty$
and $j(t)=o(t)$ as $t\to\infty$.  Then, as $t\to\infty$,
\begin{eqnarray}\label{clt22}
&&\left(\frac{\lfloor j(t)\rfloor^{1/2}(\lfloor j(t)u\rfloor-1)!}{({\tt s}^2{\tt m}^{-2\lfloor j(t)u\rfloor-1}t^{2\lfloor j(t)u\rfloor-1})^{1/2}}\bigg(\sum_{r\geq
1}\frac{(t-T^\ast_r)^{\lfloor j(t)u\rfloor-1} \1_{\{T^\ast_r\leq
t\}}}{(\lfloor j(t)u\rfloor-1)!{\tt m}^{\lfloor j(t)u\rfloor-1}}-\frac{t^{\lfloor j(t)u\rfloor}}{(\lfloor j(t)u\rfloor)!{\tt m}^{\lfloor j(t)u\rfloor}}\bigg)\right)_{u>0}\notag\\&~\Rightarrow~& \Bigg(\int_{[0,\,\infty)}e^{-uy}{\rm d}B(y)\Bigg)_{u>0}
\end{eqnarray}
in the $J_1$-topology on $D(0,\infty)$, where $(B(v))_{v\geq 0}$ is a standard Brownian motion, ${\tt m}=\me\xi<\infty$ (when $T^\ast=\bar T$ it is implicitly assumed that the distribution of $\xi$ is nonlattice). In particular, we have in \eqref{clt22} weak convergence of the finite-dimensional distributions.
\end{thm}
\begin{rem}
We note in passing that at the expense of a more technical argument the assumption $\me\eta<\infty$ in Theorem \ref{main5} can be replaced with $\me \eta^a<\infty$ for some $a>1/2$.
\end{rem}

The remainder of the paper is structured as follows. Some technical results are stated in Section \ref{auxstat} and then proved in Section \ref{auxpro}. Proposition \ref{1formain100} and Theorem \ref{main4} are deduced from these in Section \ref{auxded}. Section \ref{sect:flt} is devoted to proving Theorem \ref{main5}. The proofs of our main results, Theorems \ref{main3} and \ref{main100}, are given in Section \ref{sect:main}.

\section{Auxiliary tools}

\subsection{Results}\label{auxstat}

Lorden's inequality which is a classical result of renewal theory tells us that
\begin{equation}\label{lord}
U(t)-{\tt m}^{-1}t \leq c_0,\quad t\geq 0
\end{equation}
for appropriate constant $c_0>0$ whenever $\E\xi^2<\infty$. If the distribution of $\xi$ is nonlattice, one can take $c_0={\rm Var}\,\xi/\E\xi^2$, whereas if the distribution of
$\xi$ is $\delta$-lattice, \eqref{lord} holds with $c_0=2\delta/{\tt m}+{\rm Var}\,\xi/\E\xi^2$. Since $V(t)\leq U(t)$ for $t\geq 0$ we infer
\begin{equation}\label{lord1}
V(t)-{\tt m}^{-1}t \leq c_0,\quad t\geq 0.
\end{equation}
On the other hand, assuming that $\me\eta<\infty$,
\begin{eqnarray*}
V(t)-{\tt m}^{-1}t&=&\int_{[0,\,t]}(U(t-y)-{\tt m}^{-1}(t-y)){\rm d}G(y)\\&\hphantom{==}-& {\tt m}^{-1} \int_0^t (1-G(y)){\rm d}y\geq-{\tt m}^{-1}\int_0^t (1-G(y)){\rm d}y\geq
-{\tt m}^{-1}\me\eta
\end{eqnarray*}
having utilized $U(t)\geq {\tt m}^{-1}t$ for $t\geq 0$ which is a consequence of Wald's identity $t\leq \E S_{\nu(t)}={\tt m} U(t)$, where $\nu(t):=\inf\{k\in\mn: S_k>t\}$ for $t\geq 0$. Thus, we have shown that
\begin{equation}\label{lord2}
|V(t)-{\tt m}^{-1}t|\leq c,\quad t\geq 0
\end{equation}
where $c=\max(c_0, {\tt m}^{-1}\me\eta)$.

The following technical result is very important. It provides a two terms expansion of $V_j$ and, as such, enables us to reduce, in a sense, a more general case to the case \eqref{equa}.
\begin{assertion}\label{aux5000}
Under the assumptions $\me\xi^2<\infty$ and $\me\eta<\infty$,
\begin{equation}\label{ineq2}
\bigg|V_j(t)-\frac{t^j}{j!{\tt m}^j}\bigg|\leq \sum_{i=0}^{j-1}\binom{j}{i}\frac{c^{j-i}t^i}{i!{\tt m}^i},\quad
j\in\mn,~t\geq 0,
\end{equation}
where ${\tt m}=\me\xi<\infty$. Also, for each $t>0$ and positive integer $j\leq (t/(2c{\tt m}))^{1/2}$,
\begin{equation}\label{basic123}
\sum_{i=0}^{j-1}\binom{j}{i}\frac{c^{j-i}t^i}{i!{\tt m}^i}\leq \frac{2c jt^{j-1}}{(j-1)!{\tt m}^{j-1}}.
\end{equation}
In particular, as $t\to\infty$,
\begin{equation}\label{basic}
V_j(t)=\frac{t^j}{j!{\tt m}^j}+O\Big(\frac{jt^{j-1}}{(j-1)!{\tt m}^{j-1}}\Big)
\end{equation}
and
\begin{equation}\label{basic1}
\frac{jt^{j-1}}{(j-1)!{\tt m}^{j-1}}=o\Big(\frac{t^j}{j!{\tt m}^j}\Big)
\end{equation}
provided that $j=j(t)\to\infty$ and $j(t)=o(t^{1/2})$ as $t\to\infty$.
\end{assertion}

Recalling \eqref{equsta} we infer
\begin{equation}\label{aux987}
{\tt m}^{-1}t-{\tt m}^{-1}\me \eta\leq \bar V(t)={\tt m}^{-1}\int_0^t G(y){\rm d}y\leq {\tt m}^{-1}t,\quad t\geq 0.
\end{equation}
Thus,
\begin{equation}\label{aux567}
\bar V_j(t)\leq \frac{t^j}{j! {\tt m}^j}, j\in\mn,~ t\geq 0.
\end{equation}
The proof of Proposition \ref{aux5000} given below is only based on the two facts: 1) \eqref{lord2} and 2) $V_j$ is the $j$-fold convolution of $V$ with itself. Noting that $\bar V_j$ is the $j$-fold convolution of $\bar V$ with itself and invoking the left-hand inequality in \eqref{aux987} instead of \eqref{lord2} we obtain a counterpart of \eqref{ineq2}:
\begin{equation}\label{ineq2222}
0 \leq \frac{t^j}{j!{\tt m}^j}- \bar V_j(t) \leq \sum_{i=0}^{j-1}\binom{j}{i}\frac{t^i \bar c^{j-i}}{i!{\tt m}^i},\quad
j\in\mn,~t\geq 0,
\end{equation}
where $\bar c:={\tt m}^{-1}\me \eta$.

Lemmas \ref{aux50000} and \ref{mom_asy} are needed for the proofs of Proposition \ref{1formain100} and Theorem \ref{main4}, respectively.
\begin{lemma}\label{aux50000}
Assume that the distribution of $\xi$ is nonlattice, that $\me\xi^2<\infty$ and $\me\eta<\infty$.  Let $j=j(t)\to\infty$ through integers and $j(t)=o(t^{1/3})$ as
$t\to\infty$. Then
\begin{equation}\label{121}
\lim_{t\to\infty}\frac{j^{1/2}(j-1)!}{{\tt m}^{-j}t^{j-1/2}}\me |N_j(t)-\bar N_j(t)|=0,
\end{equation}
where ${\tt m}=\me\xi<\infty$, and
\begin{equation}\label{122}
\lim_{t\to\infty}\frac{j^{1/2}(j-1)!}{{\tt m}^{-j}t^{j-1/2}}\int_{[0,\,t]}\Big(\frac{(t-y)^{j-1}}{(j-1)!{\tt m}^{j-1}}-\bar V_{j-1}(t-y)\Big){\rm d}\bar V(y)=0.
\end{equation}
\end{lemma}
\begin{lemma}\label{mom_asy}
(a) Assume that \eqref{equa} holds and let $j=j(t)\to\infty$ through integers and $j(t)=o(t)$ as
$t\to\infty$.
Then
\begin{equation}\label{121120}
\me \Big(N_j(t)-\sum_{r\geq 1} V_{j-1}(t-T_r)\1_{\{T_r\leq t\}}\Big)^2~=~ O\Big(\frac{t^{2j-2}}{((j-1)!)^2}\Big),\quad t\to\infty.
\end{equation}
In particular,
\begin{equation}\label{121121}
\lim_{t\to\infty}\frac{j ((j-1)!)^2}{t^{2j-1}}\me \Big(N_j(t)-\sum_{r\geq 1} V_{j-1}(t-T_r)\1_{\{T_r\leq t\}}\Big)^2=0.
\end{equation}
(b) Assume that the distribution of $\xi$ is nonlattice, that $\me\xi^2<\infty$ and $\me\eta<\infty$. Let $j=j(t)\to\infty$ through integers and $j(t)=o(t^{1/2})$ as
$t\to\infty$. Then
\begin{equation}\label{121120a}
\me \Big(\bar N_j(t)-\sum_{r\geq 1} \bar V_{j-1}(t-\bar T_r)\1_{\{\bar T_r\leq t\}}\Big)^2~=~ O\Big(\frac{jt^{2j-2}}{((j-1)!)^2{\tt m}^{2j-2}}\Big),\quad t\to\infty.
\end{equation}
In particular,
\begin{equation}\label{121121a}
\lim_{t\to\infty}\frac{j ((j-1)!)^2}{{\tt m}^{-2j} t^{2j-1}}\me \Big(\bar N_j(t)-\sum_{r\geq 1}\bar V_{j-1}(t-\bar T_r)\1_{\{\bar T_r\leq t\}}\Big)^2=0.
\end{equation}
\end{lemma}

\subsection{Proofs of auxiliary results}\label{auxpro}

\begin{proof}[Proof of Proposition \ref{aux5000}]
By using the mathematical induction we first show that
\begin{equation}\label{ineq1}
\bigg|\int_{[0,\,t]}(t-y)^\ell{\rm d}V(y)-\frac{t^{\ell+1}}{(\ell+1){\tt m}}\bigg|\leq ct^\ell,\quad \ell\in\mn_0.
\end{equation}
When $\ell=0$, \eqref{ineq1} coincides with \eqref{lord2}. Assuming that \eqref{ineq1} holds for $\ell=k$ we
obtain
\begin{multline*}
\bigg|\int_{[0,\,t]}(t-y)^{k+1}{\rm d}V(y)-\frac{t^{k+2}}{(k+2){\tt m}}\bigg|=\bigg|(k+1)\int_0^t\bigg(\int_{[0,\,s]}(s-y)^k{\rm
d}V(y)-\frac{s^{k+1}}{(k+1){\tt m}}\bigg){\rm d}s\bigg|\\ \leq (k+1)\int_0^t cs^k {\rm d}s=ct^{k+1}
\end{multline*}
which completes the proof of \eqref{ineq1}.

To prove \eqref{ineq2} we once again use the mathematical induction. When $j=1$, \eqref{ineq2} coincides with \eqref{lord2}. Assuming that \eqref{ineq2} holds for $j=k$ and appealing to
\eqref{ineq1} we infer
\begin{eqnarray*}
&&\bigg|V_{k+1}(t)-\frac{t^{k+1}}{(k+1)!{\tt m}^{k+1}}\bigg|\\&\leq
&\int_{[0,\,t]}\bigg|V_k(t-y)-\frac{(t-y)^k}{k!{\tt m}^k}\bigg|{\rm
d}V(y)+\frac{1}{k!{\tt m}^k}\bigg|\int_{[0,\,t]}(t-y)^k{\rm d}V(y)-\frac{t^{k+1}}{(k+1){\tt m}}\bigg|\\&\leq&
\int_{[0,\,t]}\sum_{i=0}^{k-1}\binom{k}{i}\frac{c^{k-i}}{i!{\tt m}^i}(t-y)^i{\rm d}V(y)+\frac{ct^k}{k!{\tt m}^k}\\&\leq&
\sum_{i=0}^{k-1}\binom{k}{i}\frac{c^{k+1-i}t^i}{i!{\tt m}^i}+\sum_{i=0}^{k-1}\binom{k}{i}\frac{c^{k-i}t^{i+1}}{(i+1)!{\tt m}^{i+1}}+\frac{ct^k}{k!{\tt m}^k}\\&\leq&
c^{k+1}+\sum_{i=1}^{k-1}\bigg(\binom{k}{i}+\binom{k}{i-1}\bigg)\frac{c^{k+1-i}t^i}{i!{\tt m}^i}+\frac{(k+1)ct^k}{k!{\tt m}^k}=
\sum_{i=0}^k\binom{k+1}{i}\frac{c^{k+1-i}t^i}{i!{\tt m}^i}.
\end{eqnarray*}
Inequality \eqref{basic123} is a consequence of
\begin{multline*}
\frac{(j-1)!{\tt m}^{j-1}}{j t^{j-1}}\sum_{i=0}^{j-2}\binom{j}{i}\frac{c^{j-i}t^i}{i!{\tt m}^i}=\frac{t}{{\tt m} j^2}\sum_{i=0}^{j-2} \frac{j!}{i!} \binom{j}{i}\Big(\frac{c{\tt m}}{t}\Big)^{j-i}\le \frac{t}{{\tt m} j^2}\sum_{i=0}^{j-2}\Big(\frac{c{\tt m} j^2}{t}\Big)^{j-i}\\ \le \frac{t}{{\tt m} j^2}\sum_{i\geq 2} \Big(\frac{c{\tt m} j^2}{t}\Big)^i=\frac{{\tt m}(c j)^2}{t}\Big(1-\Big(\frac{c{\tt m} j^2}{t}\Big)\Big)^{-1}
\end{multline*}
and the facts that $x\mapsto x/(1-x)$ is a nondecreasing function and that $j^2\leq t/(2c{\tt m})$ by assumption.
Here, we have used
\begin{equation}\label{eq:bin}
\binom{j}{i}\leq \frac{j!}{i!}\leq j^{j-i}
\end{equation}
for the first inequality. Finally, while \eqref{basic1} is immediate, relation \eqref{basic} follows from \eqref{ineq2} and \eqref{basic123}.
\end{proof}

\begin{proof}[Proof of Lemma \ref{aux50000}] We start with \eqref{121}. Since $T_r\leq \bar T_r$, $r\in\mn$ almost surely (a.s.), we infer $N_j(t)\geq \bar N_j(t)$ for all $j\in\mn$ and all $t\geq 0$ a.s. Write, for large $t$, $$\me |N_j(t)-\bar N_j(t)|=V_j(t)-\bar V_j(t)\leq \sum_{i=0}^{j-1}\binom{j}{i}\frac{(c^{j-i}+\bar c^{j-i})t^i}{i!{\tt m}^i}\leq \frac{2(c+\bar c) jt^{j-1}}{(j-1)!{\tt m}^{j-1}}$$ having utilized \eqref{ineq2} and \eqref{ineq2222} for the first inequality, and \eqref{basic123} for the second. Now \eqref{121} follows from
\begin{equation*}
\frac{j^{1/2}(j-1)!}{{\tt m}^{-j}t^{j-1/2}}\frac{jt^{j-1}}{(j-1)!{\tt m}^{j-1}}~\sim~{\tt m} \Big(\frac{j^3}{t}\Big)^{1/2}~\to~0,\quad t\to\infty
\end{equation*}
which is justified by the assumption $j(t)=o(t^{1/3})$.

We pass to \eqref{122}. Invoking first \eqref{ineq2222} and then $\bar V(t)\leq {\tt m}^{-1}t$ for $t\geq 0$ yields
\begin{multline*}
0\leq \int_{[0,\,t]}\Big(\frac{(t-y)^{j-1}}{(j-1)!{\tt m}^{j-1}}-\bar V_{j-1}(t-y)\Big){\rm d}\bar V(y)\leq \int_0^t \sum_{i=0}^{j-2}\binom{j-1}{i}\frac{\bar c^{j-1-i}y^i}{i!{\tt m}^{i+1}}{\rm d}y=\sum_{i=1}^{j-1}\binom{j-1}{i-1}\frac{\bar c^{j-i}t^i}{i!{\tt m}^i}.
\end{multline*}
Next, we intend to prove that
\begin{equation*}
{\lim\sup}_{t\to\infty}\frac{(j-2)!{\tt m}^{j-1}}{t^{j-1}}\sum_{i=1}^{j-1}\binom{j-1}{i-1}\frac{\bar c^{j-i}t^i}{i!{\tt m}^i}\leq \bar c.
\end{equation*}
Using \eqref{eq:bin} we obtain $$\frac{(j-2)!}{i!}\binom{j-1}{i-1} \le (j-2)^{j-2-i} (j-1)^{j-i}\leq j^{2(j-1-i)},$$ whence
$$\frac{(j-2)!{\tt m}^{j-1}}{t^{j-1}}\sum_{i=1}^{j-1}\binom{j-1}{i-1}\frac{\bar c^{j-i}t^i}{i!{\tt m}^i}\leq \bar c\sum_{i=1}^{j-1}\Big(\frac{\bar c{\tt m} j^2}{t}\Big)^{j-i-1}\leq \bar c\Big(1-\frac{\bar c{\tt m} j^2}{t}\Big)^{-1}~\to~\bar c,\quad t\to\infty.$$ Finally,
\begin{equation*}
\frac{j^{1/2}(j-1)!}{{\tt m}^{-j}t^{j-1/2}}\frac{t^{j-1}}{(j-2)!{\tt m}^{j-1}}~\sim~{\tt m} \Big(\frac{j^3}{t}\Big)^{1/2}~\to~0,\quad t\to\infty
\end{equation*}
because $j(t)=o(t^{1/3})$.
\end{proof}

\begin{proof}[Proof of Lemma \ref{mom_asy}]
We recall our convention about the usage of $\ast$ as a superscript: for instance, $S_j^\ast$ denotes both $S_j$ and $\bar S_j$, $N_j^\ast(t)$ denotes both $N_j(t)$ and $\bar N_j(t)$ etc.

Put $D^\ast_j(t):={\rm Var}\, N^\ast_j(t)$ for $j\in\mn$ and $t\geq 0$. Recalling \eqref{basic1232} the following two equalities are almost immediate: for $j\geq 2$ and $t\geq 0$,
\begin{eqnarray}\label{aux5}
D^\ast_j(t)&=&\me \bigg(\sum_{r\geq 1}\big(N^{(\ast,\, r)}_{j-1}(t-T^\ast_r)- V^\ast_{j-1}(t-T^\ast_r)\big)\1_{\{T^\ast_r\leq t\}}\bigg)^2\\&+& \me\bigg(\sum_{r\geq 1}V^\ast_{j-1}(t-T^\ast_r)
\1_{\{T^\ast_r\leq t\}}-V^\ast_j(t)\bigg)^2\notag
\end{eqnarray}
and
\begin{eqnarray}\label{aux10}
\me \Big(N^\ast_j(t)-\sum_{r\geq 1} V^\ast_{j-1}(t-T^\ast_r)\1_{\{T^\ast_r\leq t\}}\Big)^2&=&\me \bigg(\sum_{r\geq 1}\big(N^{(\ast,\,r)}_{j-1}(t-T^\ast_r)- V^\ast_{j-1}(t-T^\ast_r)\big)\1_{\{T^\ast_r\leq t\}}\bigg)^2\notag\\&=&
\int_{[0,\,t]}D^\ast_{j-1}(t-y){\rm d}V^\ast(y).
\end{eqnarray}

We are going to obtain a reasonably precise upper bound for $$I^\ast_j(t):=\me\bigg(\sum_{r\geq 1}V^\ast_{j-1}(t-T^\ast_r)\1_{\{T^\ast_r\leq t\}}-V^\ast_j(t)\bigg)^2,\quad j\in\mn,\ t\geq 0$$ with the convention that $V^\ast_0(t)=1$ for $t\geq 0$. Note that $D^\ast_1(t)=I^\ast_1(t)$ for $t\geq 0$. As a preparation, we note that, for $j\in\mn$ and $t\geq 0$,
\begin{eqnarray*}
&&\E \sum_{r\geq 2} \sum_{1\leq i<r}V^\ast_{j-1}(t-T^\ast_i)\1_{\{T^\ast_i\leq t\}}V^\ast_{j-1}(t-T^\ast_r)\1_{\{T^\ast_r\leq
t\}}\\&\leq& \E \sum_{r\geq 1} \sum_{0\leq i<r} V^\ast_{j-1}(t-S^\ast_i)V^\ast_{j-1}(t-S^\ast_r)\1_{\{S^\ast_r\leq
t\}}\\&=&\E \sum_{i\geq 0}V^\ast_{j-1}(t-S^\ast_i)\big(V^\ast_{j-1}(t-S^\ast_{i+1})\1_{\{S^\ast_{i+1}\leq
t\}}+V^\ast_{j-1}(t-S^\ast_{i+2})\1_{\{S^\ast_{i+2}\leq t\}}+\ldots\big)\\&=&\E
\sum_{i\geq 0} V^\ast_{j-1}(t-S^\ast_i)\1_{\{S^\ast_i\leq
t\}}\E\big(V^\ast_{j-1}(t-S^\ast_i-\xi_{i+1})\1_{\{\xi_{i+1}\leq
t-S^\ast_i\}}\\&+&V^\ast_{j-1}(t-S^\ast_i-\xi_{i+1}-\xi_{i+2})\1_{\{\xi_{i+1}+\xi_{i+2}\leq
t-S^\ast_i\}}+\ldots|S^\ast_i\big)\\&=&\E\sum_{i\geq 0}V^\ast_{j-1}(t-S^\ast_i)\int_{[0,\,t-S^\ast_i]}V^\ast_{j-1}(t-S^\ast_i-y){\rm
d}\tilde U(y) \1_{\{S^\ast_i\leq t\}},
\end{eqnarray*}
where $\tilde U(t):=\sum_{i\geq 1}\mmp\{S_i\leq t\}=U(t)-1$ for $t\geq 0$. Hence,
\begin{eqnarray}\label{259}
&&I^\ast_j(t)=\me \sum_{r\geq 1}(V^\ast_{j-1}(t-T^\ast_r))^2\1_{\{T^\ast_r\leq t\}}+2\me\sum_{r\geq 2}\sum_{1\leq i<r} V^\ast_{j-1}(t-T^\ast_i)\1_{\{T^\ast_i\leq t\}}V^\ast_{j-1}(t-T^\ast_r)\1_{\{T^\ast_r\leq t\}}\notag \\&-&(V^\ast_j(t))^2\leq \me \sum_{r\geq 0}(V^\ast_{j-1}(t-S^\ast_r))^2\1_{\{S^\ast_r\leq t\}}\notag\\&+&2\E\sum_{i\geq 0}V^\ast_{j-1}(t-S^\ast_i)\int_{[0,\,t-S^\ast_i]}V^\ast_{j-1}(t-S^\ast_i-y){\rm
d}\tilde U(y) \1_{\{S^\ast_i\leq t\}} - (V^\ast_j(t))^2.
\end{eqnarray}
From now on we treat the parts (a) and (b) separately.

\noindent {\sc Proof of part (a)}. Assume that \eqref{equa} holds. Then $\tilde U(t)=t$ for $t\geq 0$ and $V_j(t)=t^j/(j!)$ for $j\in\mn$ and $t\geq 0$ which particularly entails
$$\int_{[0,\,t-S_i]}V_{j-1}(t-S_i-y){\rm d}\tilde U(y)=\int_{[0,\,t-S_i]}V_{j-1}(t-S_i-y){\rm d}y=V_j(t-S_i)$$ and thereupon
\begin{multline*}
I_j(t)\leq V^2_{j-1}(t)+\int_0^t V^2_{j-1}(y){\rm d}y+2V_{j-1}(t)V_j(t)+ 2\int_0^t V_{j-1}(y)V_j(y){\rm d}y-V^2_j(t)\\=\frac{t^{2j-2}}{((j-1)!)^2}
+\frac{t^{2j-1}}{((j-1)!)^2}\Big(\frac{2}{j}+\frac{1}{2j-1}\Big)\leq \frac{t^{2j-2}}{((j-1)!)^2}+ \frac{5 t^{2j-1}}{((j-1)!)^2(2j-1)}.
\end{multline*}
Using the latter formula together with \eqref{aux5} and
\eqref{aux10} we have
\begin{multline*}
D_j(t)=\int_0^t D_{j-1}(y){\rm d}y+I_j(t)\leq \int_0^t D_{j-1}(y){\rm
d}y+\frac{t^{2j-2}}{((j-1)!)^2}+\frac{5t^{2j-1}}{((j-1)!)^2(2j-1)}.
\end{multline*}
This in combination with the boundary condition $D_0(t)=0$ for $t\geq 0$ gives
$$D_j(t)\leq \sum_{i=0}^{j-1}\frac{t^{j-1+i}}{(i!)^2}\frac{(2i)!}{(j-1+i)!}+ 5\sum_{i=0}^{j-1}\frac{t^{j+i}}{(i!)^2}\frac{(2i)!}{(j+i)!},\quad j\in\mn,\, t\geq 0$$
whence, by \eqref{aux10},
\begin{multline*}
\me \bigg(\sum_{r\geq
1}\big(N^{(r)}_{j-1}(t-T_r)-V_{j-1}(t-T_r)\big)\1_{\{T_r\leq t\}}\bigg)^2 =\int_0^t D_{j-1}(y){\rm
d}y\\\leq \sum_{i=0}^{j-2}\frac{t^{j-1+i}}{(i!)^2}\frac{(2i)!}{(j-1+i)!}+ 5\sum_{i=0}^{j-2}\frac{t^{j+i}}{(i!)^2}\frac{(2i)!}{(j+i)!}\\\leq
\Big(\frac{2j-2}{t}+5\Big)\sum_{i=0}^{j-2}\frac{t^{j+i}}{(i!)^2}\frac{(2i)!}{(j+i)!},\quad
j\geq 2,\, t\geq 0.
\end{multline*}

Recalling that $j=j(t)=o(t)$ as $t\to\infty$ we claim that the left-hand side is asymptotic
to $5$ times the $(j-2)$th term of the last sum which is
$$\frac{t^{2j-2}}{((j-2)!)^2}\frac{(2j-4)!}{(2j-2)!}~\sim~\frac{1}{4}\frac{t^{2j-2}}{((j-1)!)^2},\quad
t\to\infty.$$ To prove this, it suffices to show that
$$\lim_{t\to\infty} \sum_{i=0}^{j-3}\frac{A(i,j,t)}{t^{j-i-2}}=0,$$
where
$$A(i,j,t):=\frac{(j!)^2 (2i)!}{(i!)^2 (j+i)! j^2}.$$ Using the inequality
\begin{equation*}
(2\pi n)^{1/2}(ne^{-1})^n \leq n!\leq e(2\pi n)^{1/2}(ne^{-1})^n
\quad n\in\mn
\end{equation*}
which is a consequence of the Stirling formula in the form
\begin{equation*}
n!= (2\pi n)^{1/2}(ne^{-1})^n e^{\theta_n/(12 n)}, n\in\mn,
\end{equation*}
where $\theta_n\in (0,1)$, we obtain
\begin{equation}\label{aux7}
\frac{1}{2^{1/2}e} A(i,j,t)\leq
\frac{4^i}{i^{1/2}}\frac{j^{2j-1}}{(j+i)^{j+i+1/2}e^{j-i-2}}\leq
4^i j^{1/2}\Big(\frac{j}{e}\Big)^{j-i-2}.
\end{equation}
This yields
$$\frac{1}{2^{1/2}e}\sum_{i=0}^{\lfloor j/2\rfloor -1}\frac{A(i,j,t)}{t^{j-i-2}}\leq
j^{1/2}\sum_{i=j-\lfloor j/2\rfloor-1}^{j-3}\Big(\frac{4j}{et}\Big)^i\leq
j^{1/2}\Big(\frac{4j}{et}\Big)^{j-\lfloor j/2\rfloor-1}\Big(1-\frac{4j}{et}\Big)^{-1}$$
having utilized $4^i\leq 4^{j-i-2}$ which holds for $0\leq i\leq \lfloor j/2\rfloor-1$. The right-hand side goes to zero as $t\to\infty$.
Another appeal to \eqref{aux7} gives
$$\frac{1}{2^{1/2}e}\sum_{i=\lfloor j/2\rfloor}^{j-3}\frac{A(i,j,t)}{t^{j-i-2}}\leq j^{1/2}
\Big(\frac{j}{et}\Big)^{j-[j/2]-2}\sum_{i=[j/2]}^{j-3}4^i\leq
\frac{1}{3} j^{1/2}4^{j-2}\Big(\frac{j}{et}\Big)^{j-\lfloor j/2\rfloor-2}.$$
The right-hand side converges to zero as $t\to\infty$. The proof of \eqref{121120} is complete, and \eqref{121121} is an immediate consequence.

\noindent {\sc Proof of part (b)}. In the present setting \eqref{259} reads
$$\bar I_j(t)\leq {\tt m}^{-1}\int_0^t \bar V^2_{j-1}(y){\rm d}y+2{\tt m}^{-1}\int_0^t \bar V_{j-1}(x)\int_{[0,\,x]}\bar V_{j-1}(x-y){\rm
d}\tilde U(y){\rm d}x- \bar V^2_j(t).$$ Here, we have used $\sum_{r\geq 0}\mmp\{\bar S_r\leq t\}={\tt m}^{-1}t$, $t\geq 0$. In the subsequent proof we shall repeatedly use \eqref{aux567}, that is, $$\bar V_{j-1}(t)\leq \frac{t^{j-1}}{(j-1)!{\tt m}^{j-1}},\quad j\geq 2,~ t\geq 0.$$ In view of \eqref{ineq2222}, for $t\geq 0$,
\begin{equation*}
\bar V^2_j(t)\geq \Big(\frac{t^j}{j!{\tt m}^j}-\sum_{i=0}^{j-1}\binom{j}{i}\frac{\bar c^{j-i}t^i}{i!{\tt m}^i}\Big)^2\geq \frac{t^{2j}}{(j!)^2 {\tt m}^{2j}}-2\frac{t^j}{j!{\tt m}^j}\sum_{i=0}^{j-1}\binom{j}{i}\frac{\bar c^{j-i}t^i}{i!{\tt m}^i}.
\end{equation*}
By the argument leading to \eqref{ineq1} but starting with $$|\tilde U(t)-{\tt m}^{-1}t|\leq \tilde c:=c_0+1$$ rather than \eqref{lord2} we conclude that
$$\int_{[0,\,x]}\bar V_{j-1}(x-y) {\rm d}\tilde U(y)\leq \int_{[0,\,x]}\frac{(x-y)^{j-1}}{(j-1)!{\tt m}^{j-1}}{\rm d}\tilde U(y)\leq \frac{x^j}{j!{\tt m}^j}+\frac{\tilde cx^{j-1}}{(j-1)!{\tt m}^{j-1}},$$ whence
$$2{\tt m}^{-1}\int_0^t \bar V_{j-1}(x)\int_{[0,\,x]}\bar V_{j-1}(x-y){\rm
d}\tilde U(y){\rm d}x\leq \frac{t^{2j}}{(j!)^2 {\tt m}^{2j}}+\frac{2\tilde c t^{2j-1}}{((j-1)!)^2 (2j-1){\tt m}^{2j-1}}.$$ Collecting fragments together we infer
$$\bar I_j(t)\leq  \frac{(1+2\tilde c) t^{2j-1}}{((j-1)!)^2 (2j-1){\tt m}^{2j-1}}+ 2\frac{t^j}{j!{\tt m}^j}\sum_{i=0}^{j-1}\binom{j}{i}\frac{\bar c^{j-i}t^i}{i!{\tt m}^i},\quad t\geq 0.$$

We claim that $(\bar D_j)_{j\in\mn_0}$ given for $t\geq 0$ by $\bar D_0(t)=0$ and $$\bar D_j(t)=\int_{[0,\,t]}\bar D_{j-1}(t-y){\rm d}\bar V(y)+\bar I_j(t),\quad j\in\mn$$ satisfies $$\bar D_j(t)\leq (1+2\tilde c)\sum_{i=0}^{j-1}\frac{(2i)! t^{j+i}}{(i!)^2 (j+i)!{\tt m}^{j+i}}+2\sum_{i=0}^{j-1}\frac{\big(\sum_{\ell=1}^{j-i}\frac{(2i+\ell)!}{\ell!} \bar c^\ell\big) t^{j+i}}{(i!)^2(j+i)!{\tt m}^{j+i}},~t\geq 0.$$ This holds when $j=1$ because $\bar D_1(t)=\bar I_1(t)$. Assuming that it holds for $j=k-1$ we obtain
\begin{eqnarray*}
\bar D_k(t) &\leq& \int_{[0,\,t]}\Big((1+2\tilde c)\sum_{i=0}^{k-2}\frac{(2i)! (t-y)^{k-1+i}}{(i!)^2 (k-1+i)!{\tt m}^{k-1+i}}\\&+& 2\sum_{i=0}^{k-2}\frac{\big(\sum_{\ell=1}^{k-1-i}\frac{(2i+\ell)!}{\ell!} \bar c^\ell\big) (t-y)^{k-1+i}}{(i!)^2(k-1+i)!{\tt m}^{k-1+i}}\Big) {\rm d}\bar V(y)+\bar I_k(t)\\ &\leq& \int_0^t \Big((1+2\tilde c)\sum_{i=0}^{k-2}\frac{(2i)! y^{k-1+i}}{(i!)^2 (k-1+i)!{\tt m}^{k+i}}\\&+& 2\sum_{i=0}^{k-2}\frac{\big(\sum_{\ell=1}^{k-1-i}\frac{(2i+\ell)!}{\ell!} \bar c^\ell\big)y^{k-1+i}}{(i!)^2(k-1+i)!{\tt m}^{k+i}}\Big){\rm d}y+\bar I_k(t)\\&\leq&
(1+2\tilde c)\Big(\sum_{i=0}^{k-2}\frac{(2i)! t^{k+i}}{(i!)^2 (k+i)!{\tt m}^{k+i}}+\frac{t^{2k-1}}{((k-1)!)^2 (2k-1){\tt m}^{2k-1}}\Big)\\&+&
2\Big(\sum_{i=0}^{k-2}\frac{\big(\sum_{\ell=1}^{k-1-i}\frac{(2i+\ell)!}{\ell!} \bar c^\ell\big)t^{k+i}}{(i!)^2(k+i)!{\tt m}^{k+i}}+\sum_{i=0}^{k-1}\binom{k}{i}\frac{\bar c^{k-i}t^{k+i}}{k!i!{\tt m}^{k+i}}\Big)\\
&=&(1+2\tilde c)\sum_{i=0}^{k-1}\frac{(2i)! t^{k+i}}{(i!)^2(k+i)!{\tt m}^{k+i}}+2\sum_{i=0}^{k-1}\frac{\big(\sum_{\ell=1}^{k-i}\frac{(2i+\ell)!}{\ell!} \bar c^\ell\big) t^{k+i}}{(i!)^2(k+i)!{\tt m}^{k+i}},
\end{eqnarray*}
and the claim follows.

Using now \eqref{aux10} in combination with $\bar V(t)\leq {\tt m}^{-1}t$ for $t\geq 0$ yields
\begin{multline*}
\me \bigg(\sum_{r\geq 1}\big(\bar N^{(r)}_{j-1}(t-\bar T_r)-\bar V_{j-1}(t-\bar T_r)\big)\1_{\{\bar T_r\leq t\}}\bigg)^2 =\int_0^t \bar D_{j-1}(y){\rm
d}\bar V(y)\\\leq (1+2\tilde c)\sum_{i=0}^{j-2}\frac{t^{j+i}}{(i!)^2}\frac{(2i)!}{(j+i)!{\tt m}^{j+i}}+ 2\sum_{i=0}^{j-2}\frac{\big(\sum_{\ell=1}^{j-1-i}\frac{(2i+\ell)!}{\ell!} \bar c^\ell\big) t^{j+i}}{(i!)^2(j+i)!{\tt m}^{j+i}},\quad j\geq 2,\, t\geq 0.
\end{multline*}
A slight modification (now the additional factors ${\tt m}^{-(j+i)}$ appear) of the argument used in the proof of part (a) show that the first sum grows like $\frac{t^{2j-1}}{4((j-1)!)^2{\tt m}^{2j-2}}$. Let us show that the second sum grows as its $(j-2)$th term $R_j(t):=\frac{\bar c t^{2j-2}}{(j-2)!(j-1)!{\tt m}^{2j-2}}$. Indeed, using the inequality $$\sum_{\ell=1}^{j-1-i}\frac{(2i+\ell)!}{\ell!} \bar c^\ell\leq \sum_{\ell=1}^{j-1-i}\frac{\bar c^\ell}{\ell!}(j-1+i)!\leq e^{\bar c}(j-1+i)!,$$ we obtain for large $t$ $$\sum_{i=0}^{j-3}\frac{\big(\sum_{\ell=1}^{j-1-i}\frac{(2i+\ell)!}{\ell!} \bar c^\ell\big) t^{j+i}}{R_j(t)(i!)^2(j+i)!{\tt m}^{j+i}}\leq \frac{e^{\bar c}}{\bar c j^3 (j-1)}\sum_{i=0}^{j-3}\Big(\frac{{\tt m}}{t}\Big)^{j-i-2}\Big(\frac{j!}{i!}\Big)^2\leq \frac{e^{\bar c}j}{\bar c (j-1)}
\sum_{i\geq 1} \Big(\frac{{\tt m} j^2}{t}\Big)^i$$ having utilized \eqref{eq:bin} for the last inequality. The right-hand side converges to zero because $j=j(t)=o(t^{1/2})$ as $t\to\infty$.
This proves \eqref{121120a}. Relation \eqref{121121a} follows from \eqref{121120a}.
\end{proof}

\subsection{Proofs of Proposition \ref{1formain100} and Theorem \ref{main4}}\label{auxded}

\begin{proof}[Proof of Proposition \ref{1formain100}]
In view of Markov's inequality and the Cram\'{e}r-Wold device, it suffices to show that
\begin{equation*}
\lim_{t\to\infty} \frac{\lfloor j(t) \rfloor^{1/2}(\lfloor j(t)\rfloor-1)!}{{\tt m}^{-\lfloor j(t)\rfloor}t^{\lfloor j(t)\rfloor-1/2}}\me\big|N_{\lfloor j(t)\rfloor}(t)-\bar N_{\lfloor j(t)\rfloor}(t)\big|=0
\end{equation*}
and
\begin{equation*}
\lim_{t\to\infty} \frac{\lfloor j(t) \rfloor^{1/2}(\lfloor j(t)\rfloor-1)!}{{\tt m}^{-\lfloor j(t)\rfloor}t^{\lfloor j(t)\rfloor-1/2}}\me\bigg|\sum_{r\geq 1}\Big(\frac{(t-\bar T_r)^{\lfloor j(t)\rfloor-1}}{(\lfloor j(t)\rfloor-1)!{\tt m}^{\lfloor j(t)\rfloor-1}}-\bar V_{\lfloor j(t)\rfloor-1}(t-\bar T_r)\Big)\1_{\{\bar T_r\leq t\}}\bigg|=0.
\end{equation*}
The first of these immediately follows from \eqref{121}. Since the expectation appearing in the second limit relation is equal to $$\int_{[0,\,t]}\Big(\frac{(t-y)^{\lfloor j(t)\rfloor-1}}{(\lfloor j(t)\rfloor-1)!{\tt m}^{\lfloor j(t)\rfloor-1}}-\bar V_{\lfloor j(t)\rfloor-1}(t-y)\Big){\rm d}\bar V(y),$$ the second limit relation is ensured by \eqref{122}.
\end{proof}

\begin{proof}[Proof of Theorem \ref{main4}]
In view of Markov's inequality and the Cram\'{e}r-Wold device weak convergence of the finite-dimensional distributions to the zero vector follows from
$$\lim_{t\to\infty}\frac{\lfloor j(t)\rfloor ((\lfloor j(t)\rfloor-1)!)^2}{{\tt m}^{-2\lfloor j(t)\rfloor} t^{2\lfloor j(t)\rfloor-1}}\me
\Big(N^\ast_{\lfloor j(t)\rfloor}(t)-\sum_{r\geq 1}V^\ast_{\lfloor j(t)\rfloor-1}(t-T^\ast_r)\1_{\{T^\ast_r\leq t\}}\Big)^2=0$$
which is justified by Lemma \ref{mom_asy}.
\end{proof}

\section{Proof of Theorem \ref{main5}}\label{sect:flt}

We start with a couple of lemmas.
\begin{lemma}\label{aux123}
Under the assumptions and notation of Theorem \ref{main5},
\begin{equation}\label{230}
\lim_{t\to\infty}t^{-1/2}\me|N^\ast(t)-{\tt m}^{-1}t|={\tt s}^2{\tt m}^{-3}\me |B(1)|.
\end{equation}
\end{lemma}
\begin{proof}
Assume first that $N^\ast=N$. According to part (B1) of Theorem 3.2 in \cite{Alsmeyer+Iksanov+Marynych:2017},
\begin{equation*}
\frac{N(t)-{\tt m}^{-1}\int_0^t G(y){\rm d}y}{({\tt s}^2{\tt m}^{-3}t)^{1/2}}~{\overset{{\rm
d}}\longrightarrow}~B(1),\quad t\to\infty.
\end{equation*}
By assumption, $\lim_{t\to\infty} \int_0^t (1-G(y)){\rm d}y=\me \eta<\infty$ which ensures that the last limit relation is equivalent to
\begin{equation}\label{clt}
\frac{N(t)-{\tt m}^{-1}t}{({\tt s}^2{\tt m}^{-3}t)^{1/2}}~{\overset{{\rm
d}}\longrightarrow}~B(1),\quad t\to\infty.
\end{equation}
By Lemma 4.2(b) in \cite{Gnedin+Iksanov:2020}, $\me (N(t)-V(t))^2=O(t)$ as $t\to\infty$. This in combination with \eqref{lord2} yields $\me (N(t)-{\tt m}^{-1}t)^2=O(t)$ as $t\to\infty$ which guarantees that the family $(t^{-1/2}(N(t)-{\tt m}^{-1}t))_{t\geq 1}$ is uniformly integrable. This together with \eqref{clt} completes the proof of \eqref{230} in the present setting.

Assume now that $N^\ast=\bar N$. According to \eqref{lord1} and \eqref{aux987}, $\me |N(t)-\bar N(t)|=V(t)-\bar V(t)\leq c_0+{\tt m}^{-1}\me \eta$ for all $t\geq 0$. The claim follows from this, the already proved relation \eqref{230} with $N^\ast=N$ and the triangle inequality.
\end{proof}

\begin{lemma}\label{lem:flt}
Under the assumptions and notation of Theorem \ref{main5}, as $t\to\infty$,
\begin{equation}\label{eq:flt}
\Big(\frac{N^\ast(ut)-{\tt m}^{-1}ut}{({\tt s}^2 {\tt m}^{-3}t)^{1/2}}\Big)_{u\geq 0}~\Rightarrow~ (B(u))_{u\geq 0}
\end{equation}
in the $J_1$-topology on $D[0,\infty)$.
\end{lemma}
\begin{proof}
When $N^\ast=N$, this follows from part (B1) of Theorem 3.2 in \cite{Alsmeyer+Iksanov+Marynych:2017} and the reasoning given right before formula \eqref{clt}.

Assume now that $N^\ast=\bar N$. Then \eqref{eq:flt} is ensured by the limit relation: for all $u\in [0,T]$
$$t^{-1/2}\sup_{u\in [0,\,T]}(N(tu)-\bar N(tu)) \stackrel{\rm{\mathbb{P}}}{\longrightarrow} 0, \quad t \to \infty,$$ where, as usual, $\stackrel{\rm{\mathbb{P}}}{\longrightarrow}$ denotes convergence in probability. The latter is a consequence of the equality
$\bar N(t)=N(t-\bar S_0)$, $t\geq 0$ (put $N(t)=0$ for $t<0$), the relation obtained in Proposition 3.3 of \cite{Alsmeyer+Iksanov+Marynych:2017} $$t^{-1/2}\sup_{u \in [0,\,T]}(N(ut)-N(ut-h))\stackrel{\rm{\mathbb{P}}}{\longrightarrow} 0, \quad t \to
\infty$$ which holds for any positive $h$ and $T$, and the fact
that $\bar S_0$ is independent of $(N(t))_{t\in\mr}$.
\end{proof}

\begin{proof}[Proof of Theorem \ref{main5}]

Our subsequent argument follows the standard path: we shall prove weak convergence of the finite-dimensional distributions and then tightness.

\noindent {\sc Proof of the finite-dimensional distributions in \eqref{clt22}}. As usual, we shall use the Cram\'{e}r-Wold device. Namely, we intend to show that for
any $\ell\in\mn$, any real $\alpha_1,\ldots, \alpha_\ell$ and any
$0<u_1<\ldots<u_\ell<\infty$, as $t\to\infty$,
\begin{equation}\label{fidi}
\sum_{i=1}^\ell \alpha_i\frac{j^{1/2}(\lfloor j u_i\rfloor-1)! Z(j u_i,
t)}{({\tt s}^2{\tt m}^{-2\lfloor j u_i\rfloor-1}t^{2\lfloor ju_i\rfloor-1})^{1/2}}~{\overset{{\rm
d}}\longrightarrow}~\sum_{i=1}^\ell \alpha_i u_i\int_0^\infty B(y)e^{-u_i y}{\rm d}y,
\end{equation}
where $$Z(ju, t):=\sum_{r\geq
1}\frac{(t-T^\ast_r)^{\lfloor ju\rfloor-1}}{(\lfloor ju\rfloor-1)!{\tt m}^{\lfloor j u\rfloor-1}}\1_{\{T^\ast_r\leq
t\}}-\frac{t^{\lfloor ju\rfloor}}{(\lfloor j u\rfloor)!{\tt m}^{\lfloor ju\rfloor}}.$$ To ease notation,
here and hereafter, we write $j$ for $j(t)$.

We have for any $u, T>0$ and sufficiently large $t$
\begin{eqnarray*}
\frac{j^{1/2}(\lfloor j u\rfloor-1)! Z(ju,t)}{({\tt s}^2{\tt m}^{-2\lfloor j u\rfloor-1}t^{2\lfloor ju\rfloor-1})^{1/2}}&=&\frac{j^{1/2}}{({\tt s}^2{\tt m}^{-3} t^{2\lfloor ju\rfloor-1})^{1/2}}
\int_{[0,\,t]}(t-y)^{\lfloor ju\rfloor-1}{\rm d}(N^\ast(y)-{\tt m}^{-1}y)\\&=&\frac{j^{1/2}(\lfloor j u\rfloor-1)}{({\tt s}^2{\tt m}^{-3}t^{2\lfloor ju\rfloor-1})^{1/2}}\bigg(\int_0^{Tt/j}(N^\ast(y)-{\tt m}^{-1}y)(t-y)^{\lfloor ju\rfloor -2}{\rm d}y\\&+&\int_{Tt/j}^t (N^\ast(y)-{\tt m}^{-1}y)(t-y)^{\lfloor ju\rfloor-2}{\rm d}y\bigg)\\&=& \frac{\lfloor j u\rfloor-1}{j}\int_0^T \frac{N^\ast(yt/j)-{\tt m}^{-1}yt/j}{({\tt s}^2{\tt m}^{-3}t/j)^{1/2}}\big(1-\frac{y}{j}\big)^{\lfloor ju\rfloor-2}{\rm d}y\\&+&
\frac{j^{1/2}(\lfloor ju\rfloor-1)}{({\tt s}^2{\tt m}^{-3}t^{2\lfloor ju\rfloor-1})^{1/2}}\int_{Tt/j}^t
(N^\ast(y)-{\tt m}^{-1}y)(t-y)^{\lfloor j u\rfloor-2}{\rm d}y.
\end{eqnarray*}
By Lemma \ref{lem:flt}, $$\Big(\frac{N^\ast(ut/j)-{\tt m}^{-1}ut/j}{({\tt s}^2 {\tt m}^{-3}t/j)^{1/2}}\Big)_{u\geq 0}~\Rightarrow~ (B(u))_{u\geq 0}$$
in the $J_1$-topology on $D[0,\infty)$. Here, we have used the assumption $t/j(t)\to\infty$. By Skorokhod's
representation theorem there exist versions $\widehat N_t$ and $\widehat B$ of $((N^\ast(ut/j)-{\tt m}^{-1}ut/j)/({\tt s}^2{\tt m}^{-3}t/j)^{1/2})_{u\geq 0}$ and $B$, respectively such that
\begin{equation}\label{cs}
\lim_{t\to\infty}\sup_{y\in [0,\,T]}\bigg|\widehat N_t(y)-\widehat B(y)\bigg|=0\quad\text{a.s.}
\end{equation}
for all $T>0$. This in combination with
$$\lim_{t\to\infty}\sup_{y\in [0,\,T]} \Big|\big(1-\frac{y}{j(t)}\big)^{\lfloor j (t)u\rfloor-2}-e^{-uy}\Big|=0$$ yields
$$\lim_{t\to\infty}\sup_{y\in [0,\,T]}\Big|\widehat{N}_t(y)\big(1-\frac{y}{j}\big)^{\lfloor ju\rfloor-2}-\widehat{B}(y)e^{-uy}\Big|=0\quad \text{a.s.}$$
Thus,
$$\lim_{t\to\infty}\sum_{i=1}^\ell \alpha_i\frac{\lfloor ju_i\rfloor-1}{j} \int_0^T \widehat N_t(y)\big(1-\frac{y}{j}\big)^{\lfloor ju_i\rfloor-2}{\rm d}y=\sum_{i=1}^\ell \alpha_i
u_i\int_0^T \widehat{B}(y)e^{-u_iy}{\rm d}y\quad \text{a.s.}$$ and
thereupon
$$\sum_{i=1}^\ell \alpha_i \frac{\lfloor j u_i\rfloor-1}{j}\int_0^T
\frac{N^\ast(yt/j)-{\tt m}^{-1}yt/j}{({\tt s}^2{\tt m}^{-3}t/j)^{1/2}}\big(1-\frac{y}{j}\big)^{\lfloor j u_i\rfloor-2}{\rm
d}y~{\overset{{\rm
d}}\longrightarrow} ~\sum_{i=1}^\ell \alpha_i u_i\int_0^T B(y)e^{-u_iy}{\rm
d}y$$ as $t\to\infty$. Since $\lim_{T\to\infty}\sum_{i=1}^\ell \alpha_i u_i \int_0^T
B(y)e^{-u_iy}{\rm d}y=\sum_{i=1}^\ell \alpha_i u_i \int_0^\infty
B(y)e^{-u_iy}{\rm d}y$ a.s.\ it remains to prove that
$$\lim_{T\to\infty}{\lim\sup}_{t\to\infty}\,\mmp\bigg\{\bigg|\sum_{i=1}^\ell\alpha_i \frac{j^{1/2}(\lfloor j u_i\rfloor-1)}{t^{\lfloor ju_i\rfloor-1/2}}\int_{Tt/j}^t
(N^\ast(y)-{\tt m}^{-1}y)(t-y)^{\lfloor ju_i\rfloor-2}{\rm d}y\bigg|>\varepsilon\bigg\}=0$$ for all $\varepsilon>0$. In view
of Markov's inequality and the fact that $\me |N^\ast(y)-{\tt m}^{-1}y|\sim
{\tt s}{\tt m}^{-3/2} \me |B(1)|y^{1/2}$ as $y\to\infty$ (see Lemma \ref{aux123}) the latter is a consequence of
$$\lim_{T\to\infty}{\lim\sup}_{t\to\infty}\,\frac{j^{1/2}(\lfloor ju\rfloor-1)}{t^{\lfloor ju\rfloor-1/2}}\int_{Tt/j}^t
y^{1/2}(t-y)^{\lfloor j u\rfloor-2}{\rm d}y=0$$ for $u>0$. To justify it,
observe that
$$\frac{j^{1/2}(\lfloor ju\rfloor-1)}{t^{\lfloor j u\rfloor-1/2}}\int_{Tt/j}^t
y^{1/2}(t-y)^{\lfloor j u\rfloor-2}{\rm d}y=\frac{\lfloor j u\rfloor-1}{j}\int_T^j
y^{1/2}\big(1-\frac{y}{j}\big)^{\lfloor j u\rfloor-2}{\rm d}y~\to~
u\int_T^\infty y^{1/2}e^{-uy}{\rm d}y$$ as $t\to\infty$ by
Lebesgue's dominated convergence theorem. The proof of
\eqref{fidi} is complete. For later use, we note that exactly the
same argument leads to
\begin{equation}\label{intermediate}
\int_0^j \bigg|\frac{N^\ast(yt/j)-{\tt m}^{-1}yt/j}{({\tt s}^2{\tt m}^{-3}t/j)^{1/2}}\bigg|\big(1-\frac{y}{j}\big)^{\lfloor j u\rfloor-2}(1+y){\rm d}y~{\overset{{\rm
d}}\longrightarrow}~\int_0^\infty |B(y)|e^{-uy}(1+y){\rm d}y
\end{equation}
for $u>0$, as $t\to\infty$.

\noindent {\sc Proof of tightness in \eqref{clt22}}. By Theorem 15.5 in
\cite{Billingsley:1968} it suffices to show that for any
$0<a<b<\infty$, $\varepsilon>0$ and $\gamma\in (0,1)$ there exist
$t_0>0$ and $\delta>0$ such that
\begin{equation}\label{tight}
\mmp\bigg\{\sup_{a\leq u,v\leq b, |u-v|\leq \delta}\,\bigg|\frac{j^{1/2}(\lfloor ju\rfloor-1)! Z(ju,
t)}{({\tt s}^2{\tt m}^{-2\lfloor ju\rfloor-1}t^{2\lfloor ju\rfloor-1})^{1/2}}-\frac{j^{1/2}(\lfloor jv\rfloor-1)!
Z(jv,t)}{({\tt s}^2{\tt m}^{-2\lfloor jv\rfloor-1}t^{2\lfloor j v\rfloor-1})^{1/2}}\bigg|>\varepsilon\bigg\}\leq \gamma
\end{equation}
for all $t\geq t_0$. As a preparation for the proof of
\eqref{tight}, let us note that for $a\leq u,v\leq b$ such that
$|u-v|\leq \delta$, $y\in [0,j]$ and large enough $j$ we have
\begin{eqnarray*}
&&\bigg|\frac{\lfloor j u\rfloor-1}{j}\big(1-\frac{y}{j}\big)^{\lfloor ju\rfloor-2}-\frac{\lfloor j v\rfloor-1}{j}\big(1-\frac{y}{j}\big)^{\lfloor j v\rfloor-2}\bigg|\\&=&\big(1-\frac{y}{j}\big)^{\lfloor j (u\wedge
v)\rfloor-2} \bigg|\frac{\lfloor j (u\vee v)\rfloor-\lfloor j(u\wedge v)\rfloor}{j}\big(1-\frac{y}{j}\big)^{\lfloor j(u\vee v)\rfloor-\lfloor j(u\wedge
v)\rfloor}\\&-&\frac{\lfloor j (u\wedge v)\rfloor-1}{j}\big(1-\big(1-\frac{y}{j}\big)^{\lfloor j(u\vee v)\rfloor-\lfloor j (u\wedge
v)\rfloor}\big)\bigg|\\&\leq&
\big(1-\frac{y}{j}\big)^{\lfloor ja\rfloor-2}\bigg(\frac{\lfloor j(u\vee
v)\rfloor-\lfloor j (u\wedge v)\rfloor}{j}+b\frac{\lfloor j(u\vee v)\rfloor-\lfloor j (u\wedge
v)\rfloor}{j}y\bigg)\\&\leq&
C|u-v|\big(1-\frac{y}{j}\big)^{\lfloor j a\rfloor-2}(1+y)\leq C\delta
\big(1-\frac{y}{j}\big)^{\lfloor ja\rfloor-2}(1+y)
\end{eqnarray*}
for appropriate constant $C>0$. With this at hand
\begin{eqnarray*}
&&\sup_{a\leq u,v\leq b, |u-v|\leq
\delta}\,\bigg|\frac{j^{1/2}(\lfloor ju\rfloor-1)! Z(ju,t)}{({\tt s}^2{\tt m}^{-2\lfloor ju\rfloor-1}t^{2\lfloor ju\rfloor-1})^{1/2}}-\frac{j^{1/2}(\lfloor j v\rfloor-1)!
Z(jv,t)}{({\tt s}^2{\tt m}^{-2\lfloor jv\rfloor-1}t^{2\lfloor j v\rfloor-1})^{1/2}}\bigg|\\&=&
\sup_{a\leq u,v\leq b, |u-v|\leq \delta}\,\bigg|\int_0^j
\frac{N^\ast(yt/j)-{\tt m}^{-1}yt/j}{({\tt s}^2{\tt m}^{-3}t/j)^{1/2}}\Big(\frac{\lfloor j u\rfloor-1}{j}\big(1-\frac{y}{j}\big)^{\lfloor j u\rfloor-2}\\&-&\frac{\lfloor j v\rfloor-1}{j}\big(1-\frac{y}{j}\big)^{\lfloor j v\rfloor-2}\Big){\rm
d}y\bigg|\\&\leq& C\delta \int_0^j
\bigg|\frac{N^\ast(yt/j)-{\tt m}^{-1}yt/j}{({\tt s}^2{\tt m}^{-3}t/j)^{1/2}}\bigg|\big(1-\frac{y}{j}\big)^{\lfloor j a\rfloor-2}(1+y){\rm d}y.
\end{eqnarray*}
Invoking \eqref{intermediate} and choosing $\delta$ sufficiently small we obtain \eqref{tight}. The proof of Theorem \ref{main5} is complete.
\end{proof}

\section{Proofs of the main results}\label{sect:main}

As far as Theorem \ref{main3} is concerned, we use a decomposition
\begin{eqnarray*}
&&K_n(\lfloor j_n u\rfloor)-(\log n)^{\lfloor j_n u\rfloor}/(\lfloor j_n u\rfloor)!= (K_n(\lfloor j_n u\rfloor)-\rho_{\lfloor j_n u\rfloor}(n))\\&+&
\Big(\rho_{\lfloor j_n u\rfloor}(n)-\sum_{r\geq 1}\frac{(\log n-T_r)^{\lfloor j_n u\rfloor-1}}{(\lfloor j_n u-1\rfloor)!}\1_{\{T_r\leq
\log n\}}\Big)\\&+&\Big(\sum_{r\geq 1}\frac{(\log n-T_r)^{\lfloor j_n u\rfloor-1}}{(\lfloor j_n u-1\rfloor)!}\1_{\{T_r\leq
\log n\}}-\dfrac{(\log n)^{\lfloor j_n u\rfloor}}{(\lfloor j_n u\rfloor)!}\Big)=:Y_1(n,u)+Y_2(n,u)+Y_3(n,u)
\end{eqnarray*}
with a perturbed random walk $T$ generated by $(\xi,\eta)$ satisfying \eqref{equa}. We claim that the following limit relations hold: as $n\to\infty$,
\begin{equation}\label{1}
\Bigg(\frac{\lfloor j_n\rfloor^{1/2}(\lfloor j_n u\rfloor-1)!\,Y_1(n,u)}{(\log n)^{\lfloor j_n
u\rfloor-1/2}}\Bigg)_{u>0}~{\overset{{\rm f.d.d.}}\longrightarrow}~(\Theta(u))_{u>0};
\end{equation}
\begin{equation}\label{2}
\Bigg(\frac{\lfloor j_n\rfloor^{1/2}(\lfloor j_n u\rfloor-1)!\,Y_2(n,u)}{(\log n)^{\lfloor j_n
u\rfloor-1/2}}\Bigg)_{u>0}~{\overset{{\rm f.d.d.}}\longrightarrow}~(\Theta(u))_{u>0};
\end{equation}
\begin{equation}\label{3}
\Bigg(\frac{\lfloor j_n\rfloor^{1/2}(\lfloor j_n u\rfloor-1)!\,Y_3(n,u)}{(\log n)^{\lfloor j_n
u\rfloor-1/2}}\Bigg)_{u>0}~{\overset{{\rm f.d.d.}}\longrightarrow}~(R(u))_{u>0}.
\end{equation}

Formula \eqref{1} states that $K_n(\lfloor j_n u\rfloor)$ is well-approximated by $\rho_{\lfloor j_n u\rfloor}(n)$. To explain why this connection is plausible, call a box $v$ with $|v|=\lfloor j_n u\rfloor$ {\it large} if $P(v)\geq 1/n$. On the average, a large box contains $nP(v)\geq 1$ balls. Hence, the number of large boxes $\rho_{\lfloor j_n u\rfloor}(n)$ should be close to the number of occupied boxes $K_n(\lfloor j_n u\rfloor)$. Since $\rho_{\lfloor j_n u\rfloor}(n)$ is a function of the environment alone, relation \eqref{1} can also be interpreted a bit differently: out of the two sources of randomness inherent to the infinite occupancy scheme in random environment the one (randomness of environment) dominates the other (variability of sampling).

Formula \eqref{2} follows by an application of Theorem \ref{main4} in the situation that \eqref{equa} prevails, with $\log n$ replacing $t$ and any positive function $t\mapsto j(t)$ satisfying $j(\log n)=j_n$ and $j(t)=o(t)$ as $t\to\infty$, after noting that $\rho_{\lfloor j_n u\rfloor}(n)=N_{\lfloor j(\log n)u\rfloor}(\log n)$.  Similarly, a specialization of Theorem \ref{main5} to the case \eqref{equa} (so that ${\tt m}={\tt s}^2=1$) with the same replacement and the same choice of $t\mapsto j(t)$ yields \eqref{3}.

For the proof of Theorem \ref{main100} we use a decomposition
\begin{eqnarray*}
&&K_n(\lfloor j_n u\rfloor)-(\mu^{-1} \log n)^{\lfloor j_n u\rfloor}/(\lfloor j_n u\rfloor)!= \big(K_n(\lfloor j_n u\rfloor)-\rho_{\lfloor j_n u\rfloor}(n)\big)\\&+&\big(\rho_{\lfloor j_n u\rfloor}(n)-\bar N_{\lfloor j_n u\rfloor}(\log n)\big)+\Big(\bar N_{\lfloor j_n u\rfloor}(\log n)-\sum_{r\geq 1} \bar V_{\lfloor j_n u\rfloor-1}(\log n-\bar T_r)\1_{\{\bar T_r\leq \log n\}}\Big)\\&-&\sum_{r\geq 1}\Big(\frac{(\log n-\bar T_r)^{\lfloor j_n u\rfloor-1}}{(\lfloor j_n u\rfloor-1)!\mu^{\lfloor j_n u\rfloor-1}}-\bar V_{\lfloor j_n u\rfloor-1}(\log n-\bar T_r)\Big)\1_{\{\bar T_r\leq \log n\}}\\&+&\Big(\sum_{r\geq 1}\frac{(\log n-\bar T_r)^{\lfloor j_n u\rfloor-1}}{\mu^{\lfloor j_n u\rfloor-1}(\lfloor j_n u-1\rfloor)!}\1_{\{\bar T_r\leq
\log n\}}-\dfrac{(\log n)^{\lfloor j_n u\rfloor}}{\mu^{\lfloor j_n u\rfloor}(\lfloor j_n u\rfloor)!}\Big)=:Y_1(n,u)+\sum_{i=2}^5 X_i(n,u)
\end{eqnarray*}
with a (stationary) perturbed random walk $\bar T$ generated by $(\xi,\eta)\od (|\log W|, |\log(1-W))|)$ for a random variable $W$ satisfying the assumptions of Theorem \ref{main100}. It suffices to check that, as $n\to\infty$,
\begin{equation}\label{1a}
\Bigg(\frac{\lfloor j_n\rfloor^{1/2}(\lfloor j_n u\rfloor-1)!\,Y_1(n,u)}{\mu^{-\lfloor j_n u\rfloor}(\log n)^{\lfloor j_n u\rfloor-1/2}}\Bigg)_{u>0}~{\overset{{\rm f.d.d.}}\longrightarrow}~(\Theta(u))_{u>0};
\end{equation}
for $i=2,3,4$
\begin{equation}\label{2a}
\Bigg(\frac{\lfloor j_n\rfloor^{1/2}(\lfloor j_n u\rfloor-1)!\,X_i(n,u)}{\mu^{-\lfloor j_n u\rfloor}(\log n)^{\lfloor j_n
u\rfloor-1/2}}\Bigg)_{u>0}~{\overset{{\rm f.d.d.}}\longrightarrow}~(\Theta(u))_{u>0};
\end{equation}
\begin{equation}\label{3a}
\Bigg(\frac{\lfloor j_n\rfloor^{1/2}(\lfloor j_n u\rfloor-1)!\,X_5(n,u)}{(\sigma^2 \mu^{-2\lfloor j_n u\rfloor-1}(\log n)^{2\lfloor j_n
u\rfloor-1})^{1/2}}\Bigg)_{u>0}~{\overset{{\rm f.d.d.}}\longrightarrow}~(R(u))_{u>0}.
\end{equation}

Formula \eqref{2a} for $i=2$ and $i=4$ follows from Proposition \ref{1formain100} in which we replace $t$ with $\log n$ and choose any positive function $t\mapsto j(t)$ satisfying $j(\log n)=j_n$ and $j(t)=o(t^{1/3})$ as $t\to\infty$. After similar adjustments formulae \eqref{2a} with $i=3$ and \eqref{3a} are ensured by Theorems \ref{main4} and \ref{main5}, respectively. Note that ${\tt s}^2=\sigma^2$ and ${\tt m}=\mu$.

Thus, we are left with checking \eqref{1} and \eqref{1a}. We start with a preparation for both proofs. In view of Markov's inequality and the Cram\'{e}r-Wold device, it is enough to show that, for fixed $u>0$,
\begin{equation}\label{aux21}
\lim_{n\to\infty} \frac{\lfloor j_n\rfloor^{1/2}(\lfloor j_n u\rfloor-1)!\,\me\big|K_n(\lfloor j_n u\rfloor)-\rho_{\lfloor j_n u\rfloor}(n)\big|}{\mu^{-\lfloor j_n u\rfloor}(\log n)^{\lfloor j_n u\rfloor-1/2}}=0,
\end{equation}
where $\mu=1$ in the setting of Theorem \ref{main3}, that is, when $W$ has a uniform distribution on $[0,1]$.

To this end, denote by $Z_{n,v}$ the number of balls in the box $v$ when $n$ balls have been thrown. Observe that, given $(P(v))_{|v|=j}$, $Z_{n,\,v}$ has a binomial distribution with parameters $n$ and $P(v)$. This follows from the branching property of the underlying weighted branching process and the following fact: if a random variable $Z_n$ has a binomial distribution with parameters $n$ and $p_1$ and is independent of $(Z^\prime_k)_{0\leq k\leq n}$, where $Z^\prime_k$ has a binomial distribution with parameters $k$ and $p_2$ ($Z_0=0$ a.s.), then $Z^\prime_{Z_n}$ has a binomial distribution with parameters $n$ and $p_1p_2$. Write, for fixed $u>0$,
\begin{multline*}
\big|K_n(\lfloor j_n u\rfloor)-\rho_{\lfloor j_n u\rfloor}(n)\big|=\Big|\sum_{|v|=\lfloor j_n u\rfloor}\1_{\{Z_{n,\,v}\geq 1\}}-\sum_{|v|=\lfloor j_n u\rfloor}\1_{\{nP(v)\geq 1\}}\Big|\\ \leq
 \sum_{|v|=\lfloor j_n u\rfloor}\1_{\{Z_{n,\,v}\geq 1,\,nP(v)<1\}}+\sum_{|v|=\lfloor j_n u\rfloor}\1_{\{Z_{n,\,v}=0,\,nP(v)\geq 1\}}
\end{multline*}
which gives
\begin{multline*}
\me\big(|K_n(\lfloor j_n u\rfloor)-\rho_{\lfloor j_n u\rfloor}(n)|\big| (P(v))_{|v|=\lfloor j_n u\rfloor}\big)\leq \sum_{|v|=\lfloor j_n u\rfloor}(1-(1-P(v))^n)\1_{\{nP(v)<1\}}\\+
\sum_{|v|=\lfloor j_n u\rfloor}(1-P(v))^n\1_{\{nP(v)\geq 1\}}\leq n\sum_{|v|=\lfloor j_n u\rfloor} P(v)\1_{\{nP(v)<1\}}+\sum_{|v|=\lfloor j_n u\rfloor}e^{-nP(v)}\1_{\{nP(v)\geq 1\}}\\=n\int_{(n,\,\infty)}x^{-1}{\rm d}\rho_{\lfloor j_n u\rfloor}(x)+\int_{[1,\,n]}e^{-n/x}{\rm d}\rho_{\lfloor j_n u\rfloor}(x).
\end{multline*}
Here, the first inequality follows from the already mentioned fact that, given $(P(v))_{|v|=\lfloor j_n u\rfloor}$, $Z_{n,\,v}$ has a binomial distribution with parameters $n$ and $P(v)$, and the second is a consequence of $1-(1-x)^n\leq nx$ and $(1-x)^n\leq e^{-nx}$ for $x\in [0,1]$, $n\in\mn$.

\noindent {\sc Proof of \eqref{aux21} in the setting of Theorem \ref{main3}}. The assumption that $W$ has a uniform distribution on $[0,1]$ entails
$$\me \rho_{\lfloor j_n u\rfloor}(x)=\frac{(\log x)^{\lfloor j_n u\rfloor}}{\lfloor j_n u \rfloor!},\quad x\geq 1$$ (this formula, in an equivalent form $V_j(t)=t^j/(j!)$ for $j\in\mn$ and $t\geq 0$, has already appeared at the beginning of the proof of Lemma \ref{mom_asy}(a)). Hence,
\begin{multline*}
I_n:=n\me\int_{(n,\,\infty)}y^{-1}{\rm d}\rho_{\lfloor j_n u\rfloor}(y)=\frac{n}{(\lfloor j_n u\rfloor-1)!}\int_n^\infty y^{-2}(\log y)^{\lfloor j_n u\rfloor-1}{\rm d}y=
\\\frac{n}{(\lfloor j_n u\rfloor-1)!}\int_{\log n}^\infty e^{-y}y^{\lfloor j_n u\rfloor-1}{\rm d}y=\sum_{i=0}^{\lfloor j_n u\rfloor-1}\frac{(\log n)^{i}}{i!}.
\end{multline*}
At the last step, we have used the following standard facts related to the gamma distribution with parameters $k\in\mn$ and $1$ (such a distributiion is sometimes called the Erlang distribution).
While the corresponding density is $x\mapsto ((k-1)!)^{-1} x^{k-1}e^{-x}\1_{(0,\infty)}(x)$, the distribution tail is $((k-1)!)^{-1} \int_x^\infty y^{k-1}e^{-y}{\rm d}y=e^{-x} \sum_{i=0}^{k-1} (x^i/(i!))$.

Let us prove that, for fixed $u>0$,
\begin{equation}\label{aux1}
\sum_{i=0}^{\lfloor j_n u\rfloor-1}\frac{(\log n)^{i}}{i!}~\sim~\frac{(\log n)^{\lfloor j_n u\rfloor-1}}{(\lfloor j_n u\rfloor-1)!},\quad n\to\infty.
\end{equation}
Indeed,
\begin{multline*}
\sum_{i=0}^{\lfloor j_n u\rfloor-1}\frac{(\log n)^{i}}{i!}=\frac{(\log n)^{\lfloor j_n u\rfloor-1}}{(\lfloor j_n u\rfloor-1)!}\Big(1+\sum_{i=1}^{\lfloor j_n u\rfloor-1}\frac{(\lfloor j_n u\rfloor-1)\cdot\ldots\cdot(\lfloor j_n u\rfloor-i)}{(\log n)^i}\Big)\\ \leq \frac{(\log n)^{\lfloor j_n u\rfloor-1}}{(\lfloor j_n u\rfloor-1)!} \sum_{i=0}^{\lfloor j_n u\rfloor-1}\frac{\lfloor j_n u\rfloor^i}{(\log n)^i}\leq\frac{(\log n)^{\lfloor j_n u\rfloor-1}}{(\lfloor j_n u\rfloor-1)!} \frac{1}{1-\lfloor j_n u\rfloor/\log n}~\sim~\frac{(\log n)^{\lfloor j_n u\rfloor-1}}{(\lfloor j_n u\rfloor-1)!}
\end{multline*}
having utilized $j_n=o(\log n)$ for the asymptotic equivalence. Since, on the other hand, $$I_n\geq \frac{(\log n)^{\lfloor j_n u\rfloor-1}}{(\lfloor j_n u\rfloor-1)!},$$ relation \eqref{aux1} follows. We note in passing that for \eqref{aux1} to hold the assumption $j_n=o(\log n)$ is of principal importance. If, for instance, $j_n\sim c\log n$ for some $c>0$, then \eqref{aux1} is no longer true.

With \eqref{aux1} at hand, we now conclude that $$\frac{j_n^{1/2}(\lfloor j_n u\rfloor-1)!I_n}{(\log n)^{\lfloor j_n u\rfloor-1/2}}\sim \Big(\frac{j_n}{\log n}\Big)^{1/2}~\to~ 0,\quad n\to\infty.$$
Put $y_n:=\exp((\log n/j_n)^{1/4})$ and write
\begin{eqnarray}
J_n&:=&\me \int_{[1,\,n]}e^{-n/x}{\rm d}\rho_{\lfloor j_n u\rfloor}(x)=\int_{[1,\,n/y_n]}e^{-n/x}{\rm d}\me \rho_{\lfloor j_n u\rfloor}(x)+\int_{(n/y_n,\,n]}e^{-n/x}{\rm d}\me \rho_{\lfloor j_n u\rfloor}(x)\notag\\ &\leq& e^{-y_n}\me \rho_{\lfloor j_n u\rfloor}(n)+ \me\rho_{\lfloor j_n u\rfloor}(n)-\me \rho_{\lfloor j_n u\rfloor}(n/y_n).\label{aux101}
\end{eqnarray}
We have
\begin{equation}\label{aux104}
\frac{j_n^{1/2}(\lfloor j_n u\rfloor-1)!e^{-y_n}\me \rho_{\lfloor j_n u\rfloor}(n)}{(\log n)^{\lfloor j_n u\rfloor-1/2}}=\frac{j_n}{\lfloor j_n u\rfloor}\Big(\frac{\log n}{j_n}\Big)^{1/2}\exp(-\exp((\log n/j_n)^{1/4}))\to 0
\end{equation}
because $\lim_{x\to 0+}x^2\exp(-e^x)=0$ and
$$\me \rho_{\lfloor j_n u\rfloor}(n)-\me\rho_{\lfloor j_n u\rfloor}(n/y_n)=\frac{(\log n)^{\lfloor j_n u\rfloor}}{\lfloor j_n u\rfloor}\big(1-(1-\log y_n/\log n\big)^{\lfloor j_n u\rfloor})\leq
\frac{(\log n)^{\lfloor j_n u\rfloor-1}}{(\lfloor j_n u\rfloor-1)!}\log y_n.$$ Hence,
\begin{equation}\label{aux102}
\frac{j_n^{1/2}(\lfloor j_n u\rfloor-1)!(\me \rho_{\lfloor j_n u\rfloor}(n)-\me \rho_{\lfloor j_n u\rfloor}(n/y_n))}{(\log n)^{\lfloor j_n u\rfloor-1/2}}\leq \Big(\frac{j_n}{\log n}\Big)^{1/2}\Big(\frac{\log n}{j_n}\Big)^{1/4}= \Big(\frac{j_n}{\log n}\Big)^{1/4}~\to~ 0.
\end{equation}
Combining pieces together we arrive at \eqref{aux21}.

\noindent {\sc Proof of \eqref{aux21} in the setting of Theorem \ref{main100}}. By Proposition \ref{aux5000}, for large $n$, $x\geq n$ and $j=j_n=o((\log n)^{1/3})$,
\begin{equation*}
\Big|\me \rho_j(x)-\frac{(\log x)^j}{j!\mu^j}\Big|\leq \frac{2cj (\log x)^{j-1}}{(j-1)!\mu^{j-1}}.
\end{equation*}
With this at hand, integrating by parts and using the calculations from the proof of \eqref{1} we obtain
\begin{multline*}
0\leq I_n=-\me \rho_{\lfloor j_n u\rfloor}(n)+n\int_n^\infty y^{-2}\me \rho_{\lfloor j_n u\rfloor}(y){\rm d}y\leq \frac{1}{\mu^{\lfloor j_n u\rfloor}}\sum_{i=0}^{\lfloor j_n u\rfloor-1}\frac{(\log n)^{i}}{i!}\\+\frac{2c\lfloor j_n u\rfloor (\log n)^{\lfloor j_n u\rfloor-1}}{(\lfloor j_n u\rfloor-1)!\mu^{\lfloor j_n u\rfloor-1}}+n\int_n^\infty y^{-2}\frac{2c\lfloor j_n u\rfloor (\log y)^{\lfloor j_n u\rfloor-1}}{(\lfloor j_n u\rfloor-1)!\mu^{\lfloor j_n u\rfloor-1}}{\rm d}y.
\end{multline*}
According to \eqref{aux1}, as $n\to\infty$,
\begin{multline*}
\frac{n \lfloor j_n u\rfloor}{(\lfloor j_n u\rfloor-1)!\mu^{\lfloor j_n u\rfloor-1}}\int_n^\infty y^{-2} (\log y)^{\lfloor j_n u\rfloor-1}{\rm d}y=
\frac{\lfloor j_n u\rfloor}{\mu^{\lfloor j_n u\rfloor-1}}\sum_{i=0}^{\lfloor j_n u\rfloor-1}\Big(\frac{\log n}{i!}\Big)^i\\~\sim~\frac{\lfloor j_n u\rfloor(\log n)^{\lfloor j_n u\rfloor-1}}{(\lfloor j_n u\rfloor-1)! \mu^{\lfloor j_n u\rfloor-1}}.
\end{multline*}
In view of $j_n=o((\log n)^{1/3})$,
\begin{equation}\label{aux103}
\frac{j_n^{1/2}(\lfloor j_n u\rfloor-1)!}{\mu^{-\lfloor j_n u\rfloor}(\log n)^{\lfloor j_n u\rfloor-1/2}}\frac{\lfloor j_n u\rfloor (\log n)^{\lfloor j_n u\rfloor-1}}
{(\lfloor j_n u\rfloor-1)!\mu^{\lfloor j_n u\rfloor-1}}~\sim~ \mu u^{1/2}\Big(\frac{j_n^3}{\log n}\Big)^{1/2}~\to~ 0,\quad n\to\infty.
\end{equation}
Another appeal to \eqref{aux1} completes the proof of
$$\lim_{n\to\infty} \frac{j_n^{1/2}(\lfloor j_n u\rfloor-1)!I_n}{\mu^{-\lfloor j_n u\rfloor}(\log n)^{\lfloor j_n u\rfloor-1/2}}=0.$$

Using \eqref{aux101} together with \eqref{basic} we conclude that $$J_n\leq \frac{e^{-y_n}(\log n)^{\lfloor j_n u\rfloor}}{(\lfloor j_n u\rfloor)!\mu^{\lfloor j_n u\rfloor}}+ \frac{(\log n)^{\lfloor j_n u\rfloor}}{\lfloor j_n u\rfloor\mu^{\lfloor j_n u\rfloor}}\big(1-(1-\log y_n/\log n\big)^{\lfloor j_n u\rfloor})+ O\Big(\frac{\lfloor j_n u\rfloor (\log y)^{\lfloor j_n u\rfloor-1}}{(\lfloor j_n u\rfloor-1)!\mu^{\lfloor j_n u\rfloor-1}}\Big),$$ where, as before, $y_n=\exp((\log n/j_n)^{1/4})$. Now \eqref{aux104}, \eqref{aux102} and \eqref{aux103} entail
$$\lim_{n\to\infty} \frac{j_n^{1/2}(\lfloor j_n u\rfloor-1)!J_n}{\mu^{-\lfloor j_n u\rfloor}(\log n)^{\lfloor j_n u\rfloor-1/2}}=0,$$ thereby completing the proof of \eqref{aux21} in the present setting.

\vspace{5mm}

\noindent {\bf Acknowledgement}. The authors thank the anonymous referee for several useful comments. D.~Buraczewski was  partially supported by the National Science Center, Poland (grant number  2019/33/ B/ST1/00207). A.~Iksanov gratefully acknowledges the support of Ulam programme funded by the Polish national agency for academic exchange (NAWA),\newline project no. PPN/ULM/2019/1/00010/DEC/1.

\end{document}